\documentclass[11pt]{article}
\usepackage[latin1]{inputenc}
\usepackage{amsmath,amsthm,amssymb}
\usepackage{amsfonts}
\usepackage{amsmath}
\allowdisplaybreaks[4]
\usepackage{amsmath,amsthm,amssymb,amscd}
\usepackage{latexsym}
\usepackage{color}
\usepackage{graphicx}
\usepackage{mathrsfs}
\usepackage{cite}
\usepackage[toc,page]{appendix}

\makeatletter \@addtoreset{equation}{section} \makeatother

\newtheorem{theorem}{Theorem}[section]

\newtheorem{proposition}{Proposition}[section]
\newtheorem{lemma}{Lemma}[section]
\newtheorem{remark}{Remark}[section]

\newcommand{\R}{\mathbb{R}}

\begin{document}

\title
{\bf Complete classification of solutions for spin-1 Bose-Einstein condensates in $\mathbb{R}$}

\date{}

\maketitle
\begin{center}
\author{\bf{Xiao Luo}}
\footnote{Email addresses: luoxiao@hfut.edu.cn (Luo).}
\author{\bf {Cunxin Wei}}
\footnote{Email addresses: 2025111344@mail.hfut.edu.cn (Wei).}
\author{\bf{Juncheng Wei}}
\footnote{Email addresses: wei@math.cuhk.edu.hk (Wei).}
\author{\bf{Maoding Zhen}}
\footnote{Email addresses: maodingzhen@163.com (Zhen).}
\end{center}
\begin{center}
\footnotesize {1 School of Mathematics, Hefei University of Technology, Hefei, 230009, P. R. China}\\
\footnotesize {2 School of Mathematics, Hefei University of Technology, Hefei, 230009, P. R. China}\\
\footnotesize {3 Department of Mathematics, Chinese University of Hong Kong, Shatin, NT,
			Hong Kong}\\
\footnotesize {4 School of Mathematics, Hefei University of Technology, Hefei, 230009, P. R. China}\\
\end{center}

\begin{abstract}
{Spin-1 Bose-Einstein condensates provide a fundamental platform for exploring quantum magnetism and coherent spin dynamics. Their three Zeeman components are coupled by spin-exchange interactions under conserved particle number and magnetization, giving rise to rich magnetic phases and nonlinear spin structures; a rigorous classification of stationary states therefore clarifies the possible spin configurations and their physical origins. In this paper, as a continuation of our previous work \cite{LLWZ23,LLWZ24,LLWZ26},  we study the classification of  solution for the spin-1 Bose-Einstein condensates in $\mathbb{R}$, by applying the non-standard Hirota's bilinear method, the determinant formula and the general left-uniqueness strategy,  we present some novel explicit solution and give the complete classification of the solutions. In comparison with the classical
Hirota's method, the non-standard Hirota's method needs to introduce auxiliary functions
to construct Hirota's bilinear form, which essentially adds additional constraint on the wave
functions, so that it possible to deal with more complex equations. As far as we know, this is the first systematic study the complete classification of solutions to the spin-1 Bose-Einstein condensates. Our classification results yields experimentally testable predictions for spin populations, relative phases, and magnetic textures \cite{Higbie,Dan,Chang2025}

}\medskip

\emph{\bf Keywords:}  Classification, Spin-1 Bose-Einstein condensate, Gross-Pitaevskii system.\medskip

\emph{\bf AMS Subject Classification:} 35J50, 35J60, 35Q40.
\end{abstract}

\section{Introduction}
In this paper, we consider the following one-dimensional spin-1 BEC, which can be described by the following coupled Gross-Pitaevskii system

\begin{equation}\label{eq9.4}
\left\{
\begin{aligned}
&u''_1+(\mu+\lambda) u_1+(c_0+c_1)u_1^3+(c_0-c_1)u_1u_{-1}^2+(c_0+c_1)u_0^2u_1+c_1u_{-1}u_0^2=0,\\
&u''_{-1}+(\mu-\lambda) u_{-1}+(c_0+c_1)u_{-1}^3+(c_0-c_1)u_1^2u_{-1}+(c_0+c_1)u_0^2u_{-1}+c_1u_1u_0^2=0,\\
&u''_{0}+\mu  u_0+c_0u_0^3+(c_0+c_1)(u_1^2+u_{-1}^2)u_0+2c_1u_1u_{-1}u_0=0,
\end{aligned}
\right.
\end{equation}
where $\mu+\lambda,\mu-\lambda,\mu \leq 0$. Spinor Bose--Einstein condensates provide a highly controllable platform in
which superfluidity and magnetism coexist at the macroscopic quantum level.
Unlike a scalar condensate, whose order parameter is described by a single
complex field, a spin-$1$ condensate is represented by the three-component
spinor
$
\Psi=(\psi_{1},\psi_{0},\psi_{-1})^{T},
$
corresponding to the Zeeman sublevels $m_F=1,0,-1$ of a fixed hyperfine
manifold. Optical confinement allows these internal spin degrees of freedom
to remain active, leading to magnetic ordering, coherent spin dynamics,
spin domains, topological defects, and multicomponent solitary waves
\cite{Higbie,Dan,Chang2025,Kurn,HOT,IMO}.

At ultralow temperatures, two spin-$1$ bosons interact through the total-spin
channels $F_{\mathrm{tot}}=0$ and $F_{\mathrm{tot}}=2$. If $g_0$ and $g_2$
denote the corresponding interaction strengths, the density and
spin-dependent coupling constants are
$
c_0=\frac{g_0+2g_2}{3},
\
c_1=\frac{g_2-g_0}{3}.
$
Accordingly, the mean-field interaction energy takes the form
$
E_{\mathrm{int}}
=
\frac{c_0}{2}\int_{\mathbb R^N}n^2\,dx
+
\frac{c_1}{2}\int_{\mathbb R^N}|\mathbf F|^2\,dx,
$
where
$
n=\sum_{j=-1}^{1}|\psi_j|^2,
\
\mathbf F=\Psi^\dagger\mathbf f\Psi
$
are the total density and local spin density, respectively. The sign of the
spin-dependent interaction distinguishes ferromagnetic and polar regimes
and determines the preferred magnetic order of the condensate. Localized stationary states arise as constrained critical points of the
energy under fixed particle number and magnetization. The associated
Lagrange multipliers may be written as a chemical potential $\mu$ and a
spin chemical potential $\lambda$, leading to the phase-matched ansatz
$
\psi_j(x,t)=e^{-i(\mu+j\lambda)t}u_j(x),
\ j=1,0,-1.
$
Consequently, $\lambda$ describes the frequency splitting between the
outer Zeeman components and is conjugate to the conserved magnetization.
Classifying the resulting stationary states determines the possible
spin-resolved density profiles, relative phases, nodal structures, and
magnetic textures compatible with the conservation laws.

Of particular interest is the algebraically distinguished regime
$c_0=c_1$. For real-valued stationary states, introduce the symmetric matrix
\[
U=
\begin{pmatrix}
u_1 & u_0/\sqrt{2}\\
u_0/\sqrt{2} & u_{-1}
\end{pmatrix}.
\]
After a suitable nondimensionalization, the three coupled stationary
equations can be written compactly as
\[
-U''-\mu U
-\frac{\lambda}{2}
\left(\sigma_3U+U\sigma_3\right)
=
2cU^3.
\]
In one spatial dimension, the corresponding time-dependent matrix nonlinear
Schr\"odinger equation possesses an integrable point and admits exact bright
spinor solitons \cite{Ieda2004}. This matrix structure explains the existence
of determinant representations and finite algebraic constructions in the
one-dimensional problem.

Started with the seminal paper\cite{LinWei05}, nonlinear Schr\"{o}dinger system have been a hot topic in the community of nonlinear analysis in the past
fifteen years, owing to their important applications in many physical models, such as in nonlinear optics and in Bose-Einstein condensates for multi-species condensates. In these studies, by using variational methods, Lyapunov-Schmidt reduction methods or bifurcation methods, various theorems, about the existence, multiplicity and qualitative properties of nontrivial solutions of nonlinear Schr\"{o}dinger systems, have been established in the literature under various assumptions.
 Since it seems almost impossible for us to provided a complete list of scholars, we refer the readers only to \cite{TBLJNS,TBNS,BZZ20,CCW11,CCWang11,glwz,glwz1,LWYZ22,LLWZ23,LLWZ24} and the references therein.

For spin-1 BEC, the existence, asymptotic behavior and stability of solutions have been studied by many authors under certain conditions, see \cite{LLWZ23,AP} and the references therein.  Cao, Chern and Wei in \cite{CCW11} proved the existence of ground states for spin-1 BEC  by minimizing the corresponding energy in one-dimensional. Kong, Wang and Zhao \cite{AP} gave the existence and detailed asymptotic behavior of ground states for spin-1 BEC with harmonic trapping potentials in two-dimensional case. Recently, in \cite{LLWZ23,LLWZ24}, we developed an exhaustive analysis on standing waves with prescribed mass of physical states for spin-1 Bose-Einstein condensate in $\mathbb{R}^{3}$ and we give a complete description on ground states  of spin-1 Bose-Einstein condensates  with Ioffe-Pritchard magnetic field in $\mathbb{R}^{2}$ and $\mathbb{R}^{3}$. For numerical results on ground states and excited states of spin-1 BEC, we refer the reader to \cite{Higbie} and the reference therein.

Recently, when $c_{1}=0,$  based on Hirota's bilinearisation method and the associated integral system,  Guo, Luo and Wei \cite{Guo25} give a complete classification on general solutions of cubic nonlinear schr\"{o}dinger system, which
address the questions of [R. Frank, D. Gontier and M. Lewin, CMP, 2021], where the complete
classification and uniqueness results for the system were already proved for the case $n=2.$

However, for the case $c_1\neq0$, the study on the classification of solutions to \eqref{eq9.4} is absent in the literatures. In this situation, the problem becomes more difficult, one main reason is that the system \eqref{eq9.4} no longer possesses symmetry, which making it difficult to derive the associated integral system.

On the other hand, the existing results show that it is usually very useful to know the complete classification  of solutions for the system \eqref{eq9.4}, which unfortunately seem very
challenging for general situations. In addition, the standard Hirota's bilinearisation method is no longer applicable to spinor BEC.  Stimulated by these facts, the main purpose of this paper is to address the complete classification of solutions for the system \eqref{eq9.4}.

To state our main results, we first introduce some notations.  Let \(\mathbf{a} = (a_1, a_0, a_{-1}) \in \mathbb{R}^3\) be arbitrary and $$\eta_{1}=\sqrt{|\mu+\lambda|},\ \eta_{-1}=\sqrt{|\mu-\lambda|},\ \eta_{0}=\sqrt{|\mu|}.$$ Define
\begin{equation}\label{M1}
C = \begin{pmatrix}
\frac{1}{2\eta_1} & \frac{1}{\eta_1 + \eta_0} & 0 & 0 \\
\frac{1}{\eta_1 + \eta_0} & \frac{1}{2\eta_0} & 0 & 0 \\
0 & 0 & \frac{1}{2\eta_0} & \frac{1}{\eta_0 + \eta_{-1}} \\
0 & 0 & \frac{1}{\eta_0 + \eta_{-1}} & \frac{1}{2\eta_{-1}}
\end{pmatrix},
\end{equation}
as well as
\begin{equation}\label{M2}
\mathcal{A}(x) = \begin{pmatrix}
a_1 e^{\eta_1 x} & 0 & 0 & 0 \\
0 & 0 & \frac{a_0}{\sqrt{2}} e^{\eta_0 x} & 0 \\
0 & \frac{a_0}{\sqrt{2}} e^{\eta_0 x} & 0 & 0 \\
0 & 0 & 0 & a_{-1} e^{\eta_{-1} x}
\end{pmatrix}, \
R = \begin{pmatrix} 1 & 1 & 0 & 0 \\ 0 & 0 & 1 & 1 \end{pmatrix}.
\end{equation}
Here \(R\) is a fixed \(2 \times 4\) matrix, not an independent variable or an interval endpoint. Denote
\begin{equation}\label{M3}
M(x) = I_4 + c\mathcal{A}(x)C\mathcal{A}(x)C.
\end{equation}After multiplying term by term, the specific matrix of \eqref{M3} is
\begin{equation}\label{T2}
M(x) = \begin{pmatrix}
1 + \frac{ca_1^2 e^{2\eta_1 x}}{4\eta_1^2} & \frac{ca_1^2 e^{2\eta_1 x}}{2\eta_1(\eta_1 + \eta_0)} & \frac{ca_0 a_1 e^{(\eta_0+\eta_1)x}}{2\sqrt{2}\eta_0(\eta_0 + \eta_1)} & \frac{ca_0 a_1 e^{(\eta_0+\eta_1)x}}{\sqrt{2}(\eta_0 + \eta_{-1})(\eta_0 + \eta_1)} \\
\frac{ca_0^2 e^{2\eta_0 x}}{4\eta_0(\eta_0 + \eta_1)} & 1 + \frac{ca_0^2 e^{2\eta_0 x}}{8\eta_0^2} & \frac{ca_0 a_{-1} e^{(\eta_0+\eta_{-1})x}}{\sqrt{2}(\eta_0 + \eta_{-1})^2} & \frac{ca_0 a_{-1} e^{(\eta_0+\eta_{-1})x}}{2\sqrt{2}\eta_{-1}(\eta_0 + \eta_{-1})} \\
\frac{ca_0 a_1 e^{(\eta_0+\eta_1)x}}{2\sqrt{2}\eta_1(\eta_0 + \eta_1)} & \frac{ca_0 a_1 e^{(\eta_0+\eta_1)x}}{\sqrt{2}(\eta_0 + \eta_1)^2} & 1 + \frac{ca_0^2 e^{2\eta_0 x}}{8\eta_0^2} & \frac{ca_0^2 e^{2\eta_0 x}}{4\eta_0(\eta_0 + \eta_{-1})} \\
\frac{ca_0 a_{-1} e^{(\eta_0+\eta_{-1})x}}{\sqrt{2}(\eta_0 + \eta_{-1})(\eta_0 + \eta_1)} & \frac{ca_0 a_{-1} e^{(\eta_0+\eta_{-1})x}}{2\sqrt{2}\eta_0(\eta_0 + \eta_{-1})} & \frac{ca_{-1}^2 e^{2\eta_{-1} x}}{2\eta_{-1}(\eta_0 + \eta_{-1})} & 1 + \frac{ca_{-1}^2 e^{2\eta_{-1} x}}{4\eta_{-1}^2}
\end{pmatrix}.
\end{equation}
Our main results are the following:
\begin{theorem}[Complete  Classification]\label{The1}
Let $c_{1}=c_{0}=c>0,$ suppose $\mathbf{u}=(u_{1},u_{0},u_{-1})$ is a solution of \eqref{eq9.4} in $(H^1(\mathbb{R})\setminus \{0\})^{3}$, then there exists a unique vector $\mathbf{a}=(a_{1},a_{0},a_{-1})$ such that
\begin{equation}\label{T1}
U_\mathbf{a}(x) = R M(x)^{-1} \mathcal{A}(x) R^T, \quad
U_\mathbf{a}(x) = \begin{pmatrix} u_1 & u_0 / \sqrt{2} \\ u_0 / \sqrt{2} & u_{-1} \end{pmatrix}.
\end{equation}
 Each set of real parameters yields a smooth \(H^1\) solution on the entire real line, and its parameters are given by
\begin{equation}\label{H1}
a_i = \lim_{x \to -\infty} e^{-\eta_i x} u_i(x), \quad i = 1, 0, -1.
\end{equation}
  In particularly, if $\lambda=0,$ then $\eta=\sqrt{-\mu}$ and there exists a unique vector $\mathbf{a}=(a_{1},a_{0},a_{-1})$ such that
\begin{equation}\label{N4}
U_\mathbf{a} = e^{\eta x} A \left( I_2 + \frac{c}{4\eta^2} e^{2\eta x} A^2 \right)^{-1},\ \text{where} \ \ A = \begin{pmatrix} a_1 & a_0 / \sqrt{2} \\ a_0 / \sqrt{2} & a_{-1} \end{pmatrix},
\end{equation}
and \(I_2\) is the \(2 \times 2\) identity matrix. Each parameter \(A\) yields a globally smooth \(H^1\) solution, and the parameter is uniquely determined by
\begin{equation*}
A = \lim_{x \to -\infty} e^{-\eta x} U_\mathbf{a}.
\end{equation*}
\end{theorem}
\begin{remark}
The Hirota method used here is nonstandard in a precise structural sense.
Besides writing $u_i=g_i/f$, it introduces the auxiliary function $s$ through
\[
    sf=g_0^2-2g_1g_{-1},
\]
and uses the formal parameter $X$ solely to record homogeneous degree in the
amplitudes. The additional variable $s$ isolates the spin-exchange combination
and closes the coupled problem as five bilinear equations whose coefficients
can be solved recursively. This triangular structure is what makes an exact
finite cascade possible and subsequently reveals its $4\times4$ determinant
encoding. Thus the modified Hirota formulation is not merely a convenient way
to compute an explicit solution: it is the structural bridge between the
original coupled system, the finite algebraic identities, and the determinant
representation. Its role is nevertheless confined to the construction stage;
the exhaustion of all $H^1$ solutions still requires the independent
asymptotic and uniqueness arguments.
\end{remark}
\begin{remark}
The determinant formula is not used merely to exhibit a family of solutions.
The argument also assigns a unique left amplitude triple to every real-valued
$(H^1(\mathbb R))^3$ solution and proves that two solutions with the same triple
must coincide. Thus the construction map $\mathbf{a}\mapsto U_\mathbf{a}$ is shown to be both
surjective and injective. This distinction between constructing explicit
solutions and exhausting the entire solution space is a central feature of the
present work.
The middle component first supplies the decay needed to analyze the two side
components, whose equations contain the inhomogeneous spin-exchange terms
$-cu_{-1}u_0^2$ and $-cu_1u_0^2$. A scalar asymptotic lemma therefore yields a
well-defined amplitude triple even when a side amplitude vanishes. The refined
asymptotics lead to an exact Volterra representation, and a componentwise
weighted contraction estimate controls every exchange term. This is the
system-specific step that adapts the general left-uniqueness strategy to the
present spin-$1$ model without imposing a common profile or an additional
spectral-gap ratio.
\end{remark}
\begin{remark}
The condition $c_0=c_1$ is structural rather than merely technical. In matrix
form, the stationary spin-$1$ BEC system with general coupling constants can be
written as
\[
 U''+\mu U+\frac{\lambda}{2}
 \bigl(\sigma_3U+U\sigma_3\bigr)
 +2c_1U^3
 +(c_0-c_1)\operatorname{tr}(U^2)U=0.
\]
When $c_0=c_1$, the trace-coupling term disappears and the nonlinearity reduces
to the single matrix covariant $U^3$. This reduction is precisely what permits
the closed Hirota bilinear system, its finite truncation, and the associated
$4\times4$ determinant representation.

If $c_0\neq c_1$, the two cubic covariants
$
 U^3
 \ \text{and}\
 \operatorname{tr}(U^2)U
$
are generically independent. Consequently, the second term cannot be absorbed
into $U^3$ by replacing the coupling constant with an effective one. In
particular, if $U_\mathbf{a}$ denotes the determinant solution of the pure matrix-cubic
equation with coefficient $c_1$, then its residual in the general system is
$
 \mathcal R(U_\mathbf{a})
 =(c_0-c_1)\operatorname{tr}(U_\mathbf{a}^2)U_\mathbf{a},
$
which is nonzero for a generic nontrivial three-component solution. At the
bilinear level, the additional trace term generates new nonlinear products that
are not encoded by the original auxiliary function. The finite algebraic
closure is therefore lost, and the Hirota cascade is not expected to terminate
at a fixed order.

The left-asymptotic analysis and the uniqueness of a solution with prescribed
left amplitudes may still remain valid when $c_0\neq c_1$. What fails is the
surjectivity of the construction: an arbitrary amplitude triple need not
produce a global solution decaying at both ends. Instead, the admissible
amplitudes must satisfy an additional right-end matching condition and are
expected, in the nonintegrable case, to form a lower-dimensional subset of the
full amplitude space. Special scalar, fixed-polarization, or rank-one reductions
may still survive because the two cubic covariants become proportional on such
invariant subspaces. Thus, the obstruction for $c_0\neq c_1$ is the loss of the
integrable finite-rank closure, rather than a mere increase in computational
complexity; it does not rule out a classification by different, generally
non-explicit methods.
\end{remark}
\begin{remark}
If \(c_0=c_1=c\leq0\), then \(u_1\equiv u_0\equiv u_{-1}\equiv0\). Indeed, testing the three equations by \(u_1,u_0,u_{-1}\), respectively, summing, and integrating by parts over \(\mathbb R\), we obtain
\begin{equation}
\sum_{j\in\{1,0,-1\}}
\int_{\mathbb R}
\left(|u_j'|^2+\eta_j^2u_j^2\right)\,dx
=
2c\int_{\mathbb R}\operatorname{tr}(U^4)\,dx.
\label{eq:energy-c-nonpositive}
\end{equation}
Since \(U\) is real and symmetric,
\[
\operatorname{tr}(U^4)
=
\operatorname{tr}\!\left((U^2)^TU^2\right)
=
\|U^2\|_F^2
\geq0.
\]
Thus, the left-hand side of \eqref{eq:energy-c-nonpositive} is nonnegative, while the right-hand side is nonpositive. Hence both sides vanish. Since \(\eta_j>0\), we conclude that \(u_j\equiv0\) for \(j\in\{1,0,-1\}\).
\end{remark}
The preceding classification Theorem \ref{The1} treats the spectral parameters
$(\mu,\lambda)$ as prescribed and establishes a one-to-one
parametrization
$
\mathbf{a}=(a_1,a_0,a_{-1})\longmapsto U_\mathbf{a}
$
of all real-valued $(H^1(\mathbb R))^3$ solutions. More precisely, for
each individual solution $\mathbf{u}$, its left asymptotic amplitude
\[
a_i(u)=\lim_{x\to-\infty}e^{-\eta_i x}u_i(x),
\quad i=1,0,-1,
\]
is uniquely determined.

From the physical point of view, however, the natural prescribed
quantities for a spin-$1$ Bose--Einstein condensate are not the
multipliers $(\mu,\lambda)$, but the total particle number
$
\mathcal N(u)
=
\int_{\mathbb R}
\bigl(u_1^2+u_0^2+u_{-1}^2\bigr)\,dx
$
and the longitudinal magnetization
$
\mathcal M(u)
=
\int_{\mathbb R}
\bigl(u_1^2-u_{-1}^2\bigr)\,dx.
$
These two quantities are conserved by the standard spin-exchange
dynamics. In the constrained variational formulation, $\mu$ is the
Lagrange multiplier associated with the conservation of total
particle number, whereas $\lambda$ is the multiplier associated with
the conservation of longitudinal magnetization. Consequently,
$\mu$ and $\lambda$ should no longer be regarded as prescribed input
parameters; instead, they must be recovered from the prescribed
physical data
$
\mathcal N(u)=N,\ \mathcal M(u)=M.
$
The purpose of the following theorem is to determine this
complete inverse image explicitly.
Let $\eta=(\eta_1,\eta_0,\eta_{-1})\in(0,\infty)^3$ and
$a=(a_1,a_0,a_{-1})\in\mathbb R^3$.
Set
\[
r_+:=(1,1,0,0)^T,
\
r_-:=(0,0,1,1)^T,
\]
and
\[
\mathcal D_{\eta,a}(x)
:=
I_4+c\mathcal A(x)
C
\mathcal A(x)
C,
\
F_{\eta,a}(x):=\det\mathcal D_{\eta,a}(x).
\]
Then $F_{\eta,a}(x)\geq1$ for every $x\in\mathbb R$. Define
\begin{align}\label{N3}
 \mathbf{u}_{\eta,a}
=
\bigl(
 \mathbf{u}_{\eta,a}^{(1)},
 \mathbf{u}_{\eta,a}^{(0)},
 \mathbf{u}_{\eta,a}^{(-1)}
\bigr),
\end{align}
where
\[
 \mathbf{u}_{\eta,a}^{(1)}(x)
=
\frac{
r_+^T
\operatorname{adj}\bigl(\mathcal D_{\eta,a}(x)\bigr)
\mathcal A(x)r_+
}{
F_{\eta,a}(x)
},\ \
\mathbf{u}_{\eta,a}^{(0)}(x)
=
\frac{
\sqrt2\,r_+^T
\operatorname{adj}\bigl(\mathcal D_{\eta,a}(x)\bigr)
\mathcal A(x)r_-
}{
F_{\eta,a}(x)
},
\]
and\[
 \mathbf{u}_{\eta,a}^{(-1)}(x)
=
\frac{
r_-^T
\operatorname{adj}\bigl(\mathcal D_{\eta,a}(x)\bigr)
\mathcal A(x)r_-
}{
F_{\eta,a}(x)
}.
\]
Fix $c>0$, $N>0$, and $M\in\mathbb R$. Consider real-valued
solutions
$
\mathbf{u}=(u_1,u_0,u_{-1})\in H^1(\mathbb R)^3
$
of the spin-$1$ Bose-Einstein condensates
subject to
\begin{align}\label{N1}
\mathcal N(u)
:=
\int_{\mathbb R}
\left(u_1^2+u_0^2+u_{-1}^2\right)\,dx=N,\  \mathcal M(u)
:=
\int_{\mathbb R}
\left(u_1^2-u_{-1}^2\right)\,dx=M.
\end{align}
Our main results on this aspect are following:
\begin{theorem}\label{The2}
(i) For every $|M|\leq N$, $\mu=-\frac{c^2N^2}{4},\
\lambda=0,$ then up to translations, \eqref{eq9.4} and \eqref{N1} admits solutions, which satisfy
\begin{equation*}
\begin{aligned}
u_1(x)
 &=\sigma\,\frac{\sqrt c}{4}(N+M)
   \operatorname{sech}\!\left(\frac{cN}{2}(x-x_0)\right),\\
u_0(x)
 &=\sigma\varepsilon\,\frac{\sqrt c}{2\sqrt2}
   \sqrt{N^2-M^2}\,
   \operatorname{sech}\!\left(\frac{cN}{2}(x-x_0)\right),\\
u_{-1}(x)
 &=\sigma\,\frac{\sqrt c}{4}(N-M)
   \operatorname{sech}\!\left(\frac{cN}{2}(x-x_0)\right),
\end{aligned}
\end{equation*}
where $
x_0\in\mathbb R,
\
\sigma,\varepsilon\in\{1,-1\}.
$

(ii) If $M=0,\ \mu=-\frac{c^2N^2}{16},
\
\lambda=0,$ then up to translations, \eqref{eq9.4} and \eqref{N1} admits solutions, which satisfy
\begin{equation*}
\begin{aligned}
u_1=\cos^2\theta\,w_1+\sin^2\theta\,w_2,\
u_0=\sqrt2\sin\theta\cos\theta\,(w_1-w_2),\
u_{-1}=\sin^2\theta\,w_1+\cos^2\theta\,w_2,
\end{aligned}
\end{equation*}
where
$
x_1,x_2,\theta\in\mathbb R,
\
\sigma_1,\sigma_2\in\{1,-1\},
$
and
$
w_j(x)
:=
\sigma_j\frac{\sqrt c\,N}{4}
\operatorname{sech}
\left(\frac{cN}{4}(x-x_j)\right),
\ j=1,2.
$

(iii) If $M=N$, $\mu=-\frac{c^2N^2}{8}-\frac{t^2}{2},
\
\lambda=-\frac{c^2N^2}{8}+\frac{t^2}{2},$  then up to translations, \eqref{eq9.4} and \eqref{N1} admits solutions, which satisfy
\begin{equation*}
 u_1(x)=\sigma_1\frac{\sqrt c\,N}{2}
 \operatorname{sech}\!\left(\frac{cN}{2}(x-x_1)\right),
 \ u_0=u_{-1}=0,
\end{equation*}
where $x_1\in\mathbb R$, $t>0,
\
t\neq\frac{cN}{2}$ and $\sigma_1\in\{1,-1\}$.

(iv) If $M=-N$ and
$
\mu=-\frac{c^2N^2}{8}-\frac{t^2}{2},
\
\lambda=\frac{c^2N^2}{8}-\frac{t^2}{2}
$, then up to translations, \eqref{eq9.4} and \eqref{N1} admits solutions, which satisfy
\begin{equation*}
 u_1=u_0=0,\
 u_{-1}(x)=\sigma_{-1}\frac{\sqrt c\,N}{2}
 \operatorname{sech}\!\left(\frac{cN}{2}(x-x_{-1})\right),
\end{equation*}
where $x_{-1}\in\mathbb R$, $ t>0,
\
t\neq\frac{cN}{2},$ and
$\sigma_{-1}\in\{1,-1\}$.

(v) If $M=0$, $\mu=-\frac{c^2N^2}{16},
\
0<|\lambda|<\frac{c^2N^2}{16},$ then up to translations, \eqref{eq9.4} and \eqref{N1} admits solutions, which satisfy
\begin{equation*}
 u_1=u_{-1}=0,\
 u_0(x)=\sigma_0\frac{\sqrt c\,N}{2\sqrt2}
 \operatorname{sech}\!\left(\frac{cN}{4}(x-x_0)\right),
\end{equation*}
where $x_0\in\mathbb R$ and $\sigma_0\in\{1,-1\}$.

(vi) If
$
0<|M|<N,\
\eta_1=\frac c4(N+M),
\
\eta_0=\frac c4\sqrt{N^2+M^2},
\
\eta_{-1}=\frac c4(N-M), $ and
\[
\mu=-\frac{c^2}{16}(N^2+M^2),
\
\lambda=-\frac{c^2NM}{8},
\]then up to translations, \eqref{eq9.4} and \eqref{N1} admits solutions, which satisfy
\begin{equation*}
\begin{aligned}
u_1(x)&=\sigma_1\frac{\sqrt c}{4}(N+M)
 \operatorname{sech}\!\left(\frac c4(N+M)(x-x_1)\right),\ u_0(x)=0,\\
u_{-1}(x)&=\sigma_{-1}\frac{\sqrt c}{4}(N-M)
 \operatorname{sech}\!\left(\frac c4(N-M)(x-x_{-1})\right),
\end{aligned}
\end{equation*}
where $
x_1,x_{-1}\in\mathbb R,
\
\sigma_1,\sigma_{-1}\in\{1,-1\}.
$

(viii) If
$
0<M<(\sqrt2-1)N,
\
M\neq\frac N3,
$
\begin{align*}
 \eta_1=\frac{cM}{2},\ \eta_0=\frac c4(N-M), \ \eta_{-1}=\frac{c}{2\sqrt2}\sqrt{N^2-2NM-M^2},
\end{align*}
and $\mu=-\frac{c^2}{16}(N-M)^2,\ \lambda=\frac{c^2}{16}(N+M)(N-3M),$
then up to translations, \eqref{eq9.4} and \eqref{N1} admits solutions, which satisfy \eqref{N3}, where $a_1a_0\neq0$ and $a_{-1}=0.$

(vii) If
$
-(\sqrt2-1)N<M<0,
\
M\neq-\frac N3,
$
\begin{align*}
 \eta_1=\frac{c}{2\sqrt2}\sqrt{N^2+2NM-M^2},\ \eta_0=\frac c4(N+M),\ \eta_{-1}=-\frac{cM}{2},
\end{align*}
and $\mu=-\frac{c^2}{16}(N+M)^2,\
 \lambda=-\frac{c^2}{16}(N-M)(N+3M), $
then up to translations, \eqref{eq9.4} and \eqref{N1} admits solutions, which satisfy \eqref{N3}, where $a_0a_{-1}\neq0$ and $a_{1}=0.$

(ix) If
$
0<|M|<(\sqrt2-1)N,
$
\begin{equation*}
\eta_1=\frac{c}{8N}(N^2+2NM-M^2),\ \eta_0=\frac{c}{8N}(N^2+M^2), \eta_{-1}=\frac{c}{8N}(N^2-2NM-M^2),
\end{equation*}
and
\[
\mu
=
-\frac{c^2(N^2+M^2)^2}{64N^2},
\
\lambda
=
-\frac{c^2M(N^2-M^2)}{16N},
\]
then up to translations, \eqref{eq9.4} and \eqref{N1} admits solutions, which satisfy \eqref{N3}, where $a_1a_0a_{-1}\neq0$.
\end{theorem}
The constrained problem is therefore an inverse problem based on the
preceding unconstrained classification: for fixed $(c,N,M)$, determine
all triples $(\mu,\lambda,\mathbf{a})$ such that the explicitly classified
solution $U_\mathbf{a}$ satisfies
$
\mathcal N(U_\mathbf{a})=N,
\
\mathcal M(U_\mathbf{a})=M.
$
Although the amplitude $\mathbf{a}$ remains unique for each particular
solution $U_\mathbf{a}$, the constraint map
$
(\mu,\lambda,\mathbf{a})
\longmapsto
\bigl(\mathcal N(U_\mathbf{a}),\mathcal M(U_\mathbf{a})\bigr)
$
is generally not injective. Indeed, the two integral constraints are
insensitive to translations and sign changes and, on several
solution strata, depend only on the decay rates and on which
asymptotic amplitudes are nonzero. Hence a fixed pair $(N,M)$ may
correspond to several branches or even continuous families of
solutions.

\begin{remark}
\label{rem:physical-admissibility}
The condition $|M|\leq N$ is not an additional technical
assumption, but a necessary consequence of the definitions of the
total particle number and the longitudinal magnetization. Indeed,
$
N+M
=
\int_{\mathbb R}
\left(2u_1^2+u_0^2\right)\,dx\geq0,
$
and
$
N-M
=
\int_{\mathbb R}
\left(u_0^2+2u_{-1}^2\right)\,dx\geq0.
$
Hence $-N\leq M\leq N$. At the two endpoints, this argument gives
the stronger conclusions
$
M=N
\Rightarrow
u_0\equiv u_{-1}\equiv0,
$
and
$
M=-N
\Rightarrow
u_1\equiv u_0\equiv0.
$
Thus the boundary values $M=\pm N$ describe fully magnetized states
supported entirely in the $m_F=\pm1$ Zeeman component.
\end{remark}

\begin{remark}
\label{rem:exact-constraint-inversion}
The first essential feature of the Theorem \ref{The2} is that it converts the
physical constraints into a finite-dimensional inverse problem. In
the noncollision regime, the exact identities
$$
N
=
\frac{2}{c}
\left(
\eta_1\varepsilon_1
+2\eta_0\varepsilon_0
+\eta_{-1}\varepsilon_{-1}
\right),
$$
and
$$
M
=
\frac{2}{c}
\left(
\eta_1\varepsilon_1
-\eta_{-1}\varepsilon_{-1}
\right),
\
\varepsilon_i:=\mathbf 1_{\{a_i\neq0\}},
$$
show that the total particle number and magnetization depend only on
the decay rates and the active asymptotic channels, but not on the
magnitudes or signs of the nonzero amplitudes. Consequently, the unconstrained parametrization
$
(\mu,\lambda,\mathbf{a})\longmapsto U_\mathbf{a}
$
can be inverted at the level of the physical data $(N,M)$ by solving
a finite algebraic system. The resulting constraint map is naturally
set-valued: each individual solution still has a unique asymptotic
amplitude, but a fixed pair $(N,M)$ may correspond to several
different branches or continuous families of solutions.
\end{remark}
\begin{remark}
\label{rem:rank-support-stratification}
A second essential feature is the different geometric mechanism
governing the collision and noncollision regimes. When $\lambda=0$,
the three decay rates coincide, and the asymptotic data must be
described by a symmetric $2\times2$ amplitude matrix. The constrained
classification is then determined by the rank and eigendirections of
this matrix, producing distinct rank-one and rank-two families.

When $\lambda\neq0$, the decay rates are distinct and the
classification is instead determined by the nonzero pattern
$
(\varepsilon_1,\varepsilon_0,\varepsilon_{-1})
\in\{0,1\}^3\setminus\{(0,0,0)\}.
$
The seven possible support patterns give the three one-amplitude,
three two-amplitude, and one three-amplitude strata. Thus the complete
classification requires the combination of two genuinely different
mechanisms.
In particular, the collision families cannot be obtained by simply
setting $\lambda=0$ in the distinct-rate formulas.
\end{remark}

\begin{remark}
\label{rem:magnetization-phase-diagram}
The complete branch structure is governed by the dimensionless
magnetization ratio
$
m:=\frac{M}{N}\in[-1,1].
$
While the rank-one collision family exists throughout the entire
physical interval $|m|\leq1$, the genuinely coupled noncollision
families occur only in the smaller regime
$
0<|m|<\sqrt2-1.
$
The threshold $\sqrt2-1$ is intrinsic: it is exactly the value at
which one of the reconstructed decay rates ceases to be positive.
The exceptional values
$
m=\pm\frac13
$
correspond to a spectral collision $\lambda=0$ and must therefore be
removed from the appropriate two-amplitude noncollision branches.

Hence the Theorem \ref{The2} provides a complete magnetization-dependent phase
diagram rather than a finite list of isolated normalized states.
Different stationary branches may coexist for the same physical data
$(c,N,M)$ while possessing different multipliers, decay rates, and
internal structures. Selecting the physically preferred branch
requires a further constrained-energy or stability analysis and is
not part of the classification theorem itself.
\end{remark}
Finally,  we outline the proof of Theorems \ref{The1} and Theorem \ref{The2}. The argument follows the the non-standard Hirota's bilinear method, the determinant formula and the general left-uniqueness strategy.

\paragraph{Step 1: Bilinearization and finite truncation.}
Set $u_i=g_i/f$ and introduce $s$ through
\begin{equation*}
    sf=g_0^2-2g_1g_{-1}.
    \label{eq:outline-bilinear-substitution}
\end{equation*}
The system then becomes five bilinear equations for
$g_1,g_0,g_{-1},f$, and $s$. A formal parameter $X$ records the degree in the
amplitudes:
\begin{equation*}
    g_i(X)=\sum_{m\ge0}g_{i,2m+1}X^{2m+1},\qquad
    f(X)=1+\sum_{m\ge1}f_{2m}X^{2m},\qquad
    s(X)=\sum_{m\ge1}s_{2m}X^{2m}.
    \label{eq:outline-formal-series}
\end{equation*}
At each order, the new coefficients are determined by lower-order terms and the
prescribed normalization at $x=-\infty$. The series terminate at
\begin{equation*}
    g_i=g_{i,1}+g_{i,3}+g_{i,5}+g_{i,7},\qquad
    f=1+f_2+f_4+f_6+f_8,\qquad
    s=s_2+s_4+s_6.
    \label{eq:outline-finite-truncation}
\end{equation*}
This truncation is exact: all possible residual coefficients vanish through
$X^{15}$ in the numerator equations, $X^{16}$ in the denominator equation,
and $X^{14}$ in the algebraic constraint.

\paragraph{Step 2: Determinant representation.}
The finite series can be written as
\begin{equation*}
    f=\det M,\qquad
    G=R\operatorname{adj}(M)AR^T,\qquad
    M=I_4+cACAC,
    \label{eq:outline-determinant-form}
\end{equation*}
where, $\operatorname{adj}(M)$ denotes the adjugate of $M $ and
\begin{equation*}
    G=
    \begin{pmatrix}
        g_1 & g_0/\sqrt{2}\\
        g_0/\sqrt{2} & g_{-1}
    \end{pmatrix}.
    \label{eq:outline-G}
\end{equation*}
The $4\times4$ matrix $M$ is an auxiliary finite-rank object, not a fourth
physical component. Adjugate and bordered-determinant identities recover the
three numerators and $s$, including the relation
\begin{equation*}
    sf=g_0^2-2g_1g_{-1}.
    \label{eq:outline-sf}
\end{equation*}

\paragraph{Step 3: Nonsingularity and decay.}
Although $M$ is generally nonsymmetric, the Gram matrix $C$ is positive
definite. Hence
\begin{equation*}
    H=C^{1/2}AC^{1/2}
    \label{eq:outline-H}
\end{equation*}
is symmetric and
\begin{equation*}
    C^{1/2}MC^{-1/2}=I_4+cH^2,\qquad
    f=\prod_{j=1}^{4}(1+c\rho_j^2)\ge1,
    \label{eq:outline-positive-denominator}
\end{equation*}
where $\rho_j\in\mathbb R$ are the eigenvalues of $H$. Thus $f$ never
vanishes. If some amplitudes are zero, the right-end analysis is restricted to
their nonzero coordinate subspace $E$, where $A_E$ is invertible; no inverse of
the possibly singular matrix $A$ is used. It follows that
\begin{equation*}
    U_\mathbf{a}(x),\ U_\mathbf{a}'(x)=O(e^{-\eta_*x})
    \qquad (x\to+\infty),
    \label{eq:outline-right-decay}
\end{equation*}
where $\eta_*:=\min_i\eta_i$. The explicit left asymptotics then show that each
$a\in\mathbb R^3$ yields a globally smooth, real-valued
$H^1(\mathbb R)^3$ solution.

\paragraph{Step 4: Recovery of the left amplitudes.}
Let $\mathbf{u}$ be any real-valued $(H^1(\mathbb R))^3$ solution. A scalar asymptotic
lemma for
\begin{equation*}
    y''-\eta^2y=q(x)y+h(x),\quad q\in L^1,
    \label{eq:outline-scalar-asymptotic-equation}
\end{equation*}
is applied first to the homogeneous equation for $u_0$ and then to the side
equations, whose forcing terms are
\begin{equation*}
    h_1=-cu_{-1}u_0^2,\quad h_{-1}=-cu_1u_0^2.
    \label{eq:outline-side-forcing}
\end{equation*}
This gives the unique amplitudes
\begin{equation*}
    a_i=\lim_{x\to-\infty}e^{-\eta_i x}u_i(x),
    \quad i\in\{1,0,-1\},
    \label{eq:outline-left-amplitude}
\end{equation*}
and the refined estimates
\begin{equation*}
    u_i=a_i e^{\eta_i x}+O\!\left(e^{(\eta_i+2\eta_*)x}\right),\qquad
    u_i'=\eta_i a_i e^{\eta_i x}
        +O\!\left(e^{(\eta_i+2\eta_*)x}\right)
    \quad (x\to-\infty).
    \label{eq:outline-improved-left-asymptotics}
\end{equation*}
No division by $a_i$ is required, so zero leading amplitudes are included.

\paragraph{Step 5: Uniqueness from the left amplitudes.}
The refined estimates yield the Volterra formula
\begin{equation*}
    u_i(x)=a_i e^{\eta_i x}
    +\frac{1}{\eta_i}\int_{-\infty}^{x}
        \sinh\!\bigl(\eta_i(x-t)\bigr)F_i(u(t))\,dt.
    \label{eq:outline-volterra}
\end{equation*}
For two solutions $u$ and $v$ with the same amplitudes, define
\begin{equation*}
    Z(x)=\max_{i\in\{1,0,-1\}}e^{-\eta_i x}|u_i(x)-v_i(x)|.
    \label{eq:outline-weighted-difference}
\end{equation*}
The cubic terms satisfy
\begin{equation*}
    |F_i(u(x))-F_i(v(x))|
    \le C e^{(\eta_i+2\eta_*)x}Z(x),
    \quad x\le b_0.
    \label{eq:outline-nonlinear-difference}
\end{equation*}
Therefore, for
$W_b:=\sup_{x\le b}Z(x)$,
\begin{equation*}
    W_b\le \frac{C}{4\eta_*^2}e^{2\eta_*b}W_b.
    \label{eq:outline-left-contraction}
\end{equation*}
For sufficiently negative $b$, this implies $W_b=0$. The solutions coincide on
a left half-line and hence, by ODE uniqueness, on all of $\mathbb R$.

\paragraph{Step 6: Classification.}
For an arbitrary $H^1$ solution $u$, let $\mathbf{a}(u)$ be the amplitude triple from
Step~4. The explicit solution $U_{\mathbf{a}(u)}$ has the same left amplitudes, so Step~5
gives
\begin{equation*}
    u=U_{\mathbf{a}(u)}\quad\text{on }\mathbb R.
    \label{eq:outline-matching}
\end{equation*}
Conversely, every $\mathbf{a}\in\mathbb R^3$ produces a solution of \eqref{eq:outline-left-amplitude}. Thus
$\mathbf{a}\mapsto U_\mathbf{a}$ is a bijection from $\mathbb R^3$ onto the set of real-valued
$(H^1(\mathbb R))^3$ solutions.

We divide the proof of Theorem \ref{The2} into necessity, algebraic inversion, and
sufficiency.

First, every function satisfying the two constraints necessarily
obeys
$
N+M
=
\int_{\mathbb R}
\left(2u_1^2+u_0^2\right)\,dx\geq0,
$
and
$
N-M
=
\int_{\mathbb R}
\left(u_0^2+2u_{-1}^2\right)\,dx\geq0.
$
Hence $|M|\leq N$, and no constrained solution can exist when
$|M|>N$.

We next apply the unconstrained classification. For fixed
$(c,\mu,\lambda)$, every real-valued $H^1(\mathbb R)^3$ solution is
represented by a unique asymptotic amplitude. The cases
$\lambda=0$ and $\lambda\neq0$ must be treated separately because
they correspond, respectively, to coincident and distinct decay
rates.

Suppose first that $\lambda=0$. Then
$
\eta_1=\eta_0=\eta_{-1}=:\eta,
$
and the solution is parametrized by a symmetric $2\times2$ amplitude
matrix $A$. If
$
A=O\operatorname{diag}(\alpha_1,\alpha_2)O^T,
$
the exact mass-matrix identity is
\[
\int_{\mathbb R}U_A(x)^2\,dx
=
\frac{2\eta}{c}\,
O
\operatorname{diag}
\bigl(
\mathbf 1_{\{\alpha_1\neq0\}},
\mathbf 1_{\{\alpha_2\neq0\}}
\bigr)
O^T.
\]
When $\operatorname{rank}A=1$, the two constraints give
$
\eta=\frac{cN}{2},
\
\mu=-\frac{c^2N^2}{4},
$
and determine the eigendirection through
$
v_1^2=\frac{N+M}{2N},
\
v_2^2=\frac{N-M}{2N}.
$
Substitution into the scalar $\operatorname{sech}$ profile yields
the rank-one formula stated in the Theorem \ref{The2}. When $\operatorname{rank}A=2$, the mass matrix is
$
\frac{2\eta}{c}I_2.
$
Consequently,
\[
M=0,
\
\eta=\frac{cN}{4},
\
\mu=-\frac{c^2N^2}{16}.
\]
Diagonalizing $A$ and writing its two scalar eigenprofiles as
$w_1,w_2$ gives the rank-two family.

Now suppose that $\lambda\neq0$. Introduce
$
\eta_1=\sqrt{-\mu-\lambda},
\
\eta_0=\sqrt{-\mu},
\
\eta_{-1}=\sqrt{-\mu+\lambda}.
$
These rates are positive and satisfy
$
\eta_1^2+\eta_{-1}^2=2\eta_0^2.
$
For the determinant solution $U_\mathbf{a}$, let
$
\varepsilon_i:=\mathbf 1_{\{a_i\neq0\}},
\ i=1,0,-1.
$
The exact mass identities give
$
N
=
\frac{2}{c}
\left(
\eta_1\varepsilon_1+
2\eta_0\varepsilon_0+
\eta_{-1}\varepsilon_{-1}
\right),
$
and
$
M
=
\frac{2}{c}
\left(
\eta_1\varepsilon_1-
\eta_{-1}\varepsilon_{-1}
\right).
$
Thus the constrained problem reduces to the algebraic system
\[
\eta_1\varepsilon_1+
2\eta_0\varepsilon_0+
\eta_{-1}\varepsilon_{-1}
=\frac{cN}{2},
\]
\[
\eta_1\varepsilon_1-
\eta_{-1}\varepsilon_{-1}
=\frac{cM}{2},\ \eta_1^2+\eta_{-1}^2=2\eta_0^2.
\]
There are exactly seven nonempty support patterns
$
(\varepsilon_1,\varepsilon_0,\varepsilon_{-1})
\in
\{0,1\}^3\setminus\{(0,0,0)\}.
$
Solving the preceding three equations for each support pattern gives
the three one-amplitude, three two-amplitude, and one
three-amplitude families listed in the Theorem \ref{The2}. The stated
conditions on $M/N$ are precisely the conditions ensuring
$
\eta_1,\eta_0,\eta_{-1}>0.
$
The exceptional values $M=\pm N/3$ in the mixed two-amplitude
families give $\lambda=0$ and are therefore removed from the
noncollision classification.

For the one-amplitude and two-outer-amplitude cases, the active
equations reduce to scalar equations of the form
$$
y''-\eta^2y+2cy^3=0,\ \text{or}\ y''-\eta^2y+cy^3=0.
$$
Their nonzero $H^1(\mathbb R)$ solutions are, respectively,
$$
y(x)
=
\sigma\frac{\eta}{\sqrt c}
\operatorname{sech}\bigl(\eta(x-x_*)\bigr)\ \text{and}\ y(x)
=
\sigma\frac{\sqrt2\,\eta}{\sqrt c}
\operatorname{sech}\bigl(\eta(x-x_*)\bigr).
$$
Substituting the reconstructed decay rates gives the explicit scalar
profiles in the Theorem \ref{The2}.

For the remaining coupled families, substitution of the reconstructed
rates into the determinant formula gives
$
U_a
=
R\mathcal D_{\eta,a}^{-1}
\mathcal AR^T.
$
Since
$
\det\mathcal D_{\eta,a}(x)\geq1,
$
these profiles are globally smooth and have no singularities. The
adjugate identity
$
\mathcal D_{\eta,a}^{-1}
=
\frac{
\operatorname{adj}(\mathcal D_{\eta,a})
}{
\det\mathcal D_{\eta,a}
}
$
then gives the three componentwise formulas for
$u_1,u_0,u_{-1}$.

Finally, the exact mass identities verify that every listed profile
satisfies the prescribed values of $N$ and $M$. Conversely, the
unconstrained classification assigns a unique asymptotic amplitude
to every $H^1$ solution, while the rank decomposition for
$\lambda=0$ and the seven support patterns for $\lambda\neq0$
exhaust all possible amplitudes. Hence no additional constrained
families can occur.

The paper is organized as follows. In Section \ref{S1}, we introduce some preliminary results. In Section \ref{S2}, we introduce the Cascade functions
and determinant representation. In section \ref{S3}, we prove Theorem \ref{The1}. Finally, Theorem \ref{The2} will be proved in section \ref{S4}.

\section{Preliminaries  }\label{S1}
In this section, we give some preliminaries which are useful for the rest of the paper. We first write out the complete scalar proof used in Lemma 3.1 of \cite{Guo26}, and then check item by item whether the self-spin system satisfies its assumptions, in order to avoid mistakenly applying the non-zero amplitude conclusions unique to that reference to our system.
\begin{lemma} \label{LE7}[Lemma 3.1 in \cite{Guo26}]
Let \(\eta > 0\), \(q \in C((-\infty, b)) \cap L^1(-\infty, b)\), \(h \in C((-\infty, b))\), and let there exist \(\sigma > 0\) such that \(h(x) = O(e^{(\eta+\sigma)x})\) at the left end. If \(y \in H^1(-\infty, b) \cap C^2(-\infty, b)\) satisfies
\begin{equation*}
y'' - \eta^2 y = q(x)y + h(x),
\end{equation*}
then there exists a unique \(a \in \mathbb{R}\) such that
\begin{equation*}
y = ae^{\eta x} + o(e^{\eta x}), \quad y' = \eta ae^{\eta x} + o(e^{\eta x}).
\end{equation*}
Only when additionally \(h \equiv 0\), can we further deduce that \(a = 0 \Rightarrow y \equiv 0\).
\end{lemma}
Now, we introduce the non-standard Hirota's bilinear method.
Let $$u_{1}=\frac{g_{1}}{f},\ \ \ \ u_{0}=\frac{g_{0}}{f},\ \  u_{-1}=\frac{g_{-1}}{f},$$
then $$u''_{1}=\frac{(g''_{1}f-g_{1}f'')f^{2}-2(g'_{1}f-g_{1}f')ff'}{f^{4}},$$$$u''_{-1}=\frac{(g''_{-1}f-g_{-1}f'')f^{2}-2(g'_{-1}f-g_{-1}f')ff'}{f^{4}},$$
$$u''_{0}=\frac{(g''_{0}f-g_{0}f'')f^{2}-2(g'_{0}f-g_{0}f')ff'}{f^{4}}.$$
Substituting above equalities into system  \eqref{eq9.4}, we obtain
\begin{equation}\label{eq9.1b}
\left\{
\begin{aligned}
&\left[g''_{1}f+g_{1}f''-2g'_{1}f'+(\mu+\lambda)g_{1}f\right]f^{2}\\
&+g_{1}f\left[2(f')^{2}-2f''f+(c_{0}+c_{1})g^{2}_{1}+(c_{0}-c_{1})g^{2}_{-1}+(c_{0}+c_{1})g^{2}_{0}\right]+c_{1}g_{-1}g^{2}_{0}f=0,\\
&\left[g''_{-1}f+g_{-1}f''-2g'_{-1}f'+(\mu-\lambda)g_{-1}f\right]f^{2}\\
&+g_{-1}f\left[2(f')^{2}-2f''f+(c_{0}+c_{1})g^{2}_{-1}+(c_{0}-c_{1})g^{2}_{1}+(c_{0}+c_{1})g^{2}_{0}\right]+c_{1}g_{1}g^{2}_{0}f=0,\\
&\left[g''_{0}f+g_{0}f''-2g'_{0}f'+\mu g_{0}f\right]f^{2}\\
&+g_{0}f\left[2(f')^{2}-2f''f+c_{0}g^{2}_{0}+(c_{0}+c_{1})(g^{2}_{1}+g^{2}_{-1})+2c_{1}g_{1}g_{-1}\right]=0.\\
\end{aligned}
\right.
\end{equation}
Introduce auxiliary functions
$$sf=g^{2}_{0}-2g_{1}g_{-1},$$
then \eqref{eq9.1b} becomes
\begin{equation}\label{eq9.1c}
\left\{
\begin{aligned}
&\left[g''_{1}f+g_{1}f''-2g'_{1}f'+(\mu+\lambda)g_{1}f+c_{1}sg_{-1}\right]f^{2}\\
&+g_{1}f\left[2(f')^{2}-2f''f+(c_{0}+c_{1})g^{2}_{1}+(c_{0}+c_{1})g^{2}_{-1}+(c_{0}+c_{1})g^{2}_{0}\right]=0,\\
&\left[g''_{-1}f+g_{-1}f''-2g'_{-1}f'+(\mu-\lambda)g_{-1}f+c_{1}sg_{1}\right]f^{2}\\
&+g_{-1}f\left[2(f')^{2}-2f''f+(c_{0}+c_{1})g^{2}_{-1}+(c_{0}+c_{1})g^{2}_{1}+(c_{0}+c_{1})g^{2}_{0}\right]=0,\\
&\left[g''_{0}f+g_{0}f''-2g'_{0}f'+\mu g_{0}f-c_{1}sg_{0}\right]f^{2}\\
&+g_{0}f\left[2(f')^{2}-2f''f+(c_{0}+c_{1})(g^{2}_{0}+g^{2}_{1}+g^{2}_{-1})\right]=0,\\
\end{aligned}
\right.
\end{equation}
which implies that
\begin{equation}\label{eq9.1d}
\left\{
\begin{aligned}
&g''_{1}f+g_{1}f''-2g'_{1}f'+(\mu+\lambda)g_{1}f+c_{1}sg_{-1}=0,\\
&g''_{-1}f+g_{-1}f''-2g'_{-1}f'+(\mu-\lambda)g_{-1}f+c_{1}sg_{1}=0,\\
&g''_{0}f+g_{0}f''-2g'_{0}f'+\mu g_{0}f-c_{1}sg_{0}=0,\\
&2(f')^{2}-2f''f+(c_{0}+c_{1})(g^{2}_{0}+g^{2}_{1}+g^{2}_{-1})=0,\\
&sf=g^{2}_{0}-2g_{1}g_{-1}.\\
\end{aligned}
\right.
\end{equation}
We will use the following power series expansions

\begin{align}\label{L1}
&g_{i}=g_{i1}X+g_{i3}X^{3}+g_{i5}X^{5}+g_{i7}X^{7},\ i=-1,0,1\\\nonumber
&f=1+f_{2}X^{2}+f_{4}X^{4}+f_{6}X^{6}+f_{8}X^{8},\\\nonumber
&s=s_{2}X^{2}+s_{4}X^{4}+s_{6}X^{6},
\end{align}
where $X$ is the formal parameter, the complex functions $g_{j1},g_{j3},g_{j5},g_{j7}(j=1,0,-1),$ real functions $f_{2},\ f_{4},\ f_{6},\ f_{8}$ and complex functions $s_{2},\ s_{4},\ s_{6} $ will be determined later.

We then analyze the following cascade of equations in powers of $X$. By some complex calculations, we can obtain that
\begin{align}\label{A1}
g''_{11}+(\mu+\lambda)g_{11}=0,\ g''_{-11}+(\mu-\lambda)g_{-11}=0,\ g''_{01}+\mu g_{01}=0,
\end{align}
\begin{align}\label{A2}
\frac{c_{0}+c_{1}}{2}(g^{2}_{11}+g^{2}_{-11}+g^{2}_{01})-f''_{2}=0,\ s_{2}-g^{2}_{01}+2g_{11}g_{-11}=0,
\end{align}
\begin{equation}\label{A3}
\left\{
\begin{aligned}
&g''_{11}f_{2}+g''_{13}+g_{11}f''_{2}-2g'_{11}f'_{2}+(\mu+\lambda)(g_{11}f_{2}+g_{13})+c_{1}s_{2}g_{-11}=0,\\
&g''_{-11}f_{2}+g''_{-13}+g_{-11}f''_{2}-2g'_{-11}f'_{2}+(\mu-\lambda)(g_{-11}f_{2}+g_{-13})+c_{1}s_{2}g_{11}=0,\\
&g''_{01}f_{2}+g''_{03}+g_{01}f''_{2}-2g'_{01}f'_{2}+\mu(g_{01}f_{2}+g_{03})-c_{1}s_{2}g_{01}=0,
\end{aligned}
\right.
\end{equation}
\begin{equation}\label{A4}
\left\{
\begin{aligned}
&(f'_{2})^{2}-(f''_{2}f_{2}+f''_{4})+(c_{0}+c_{1})(g_{11}g_{13}+g_{-11}g_{-13}+g_{01}g_{03})=0,\\
&s_{2}f_{2}+s_{4}-2g_{01}g_{03}+2(g_{11}g_{-13}+g_{13}g_{-11})=0,
\end{aligned}
\right.
\end{equation}
\begin{equation}\label{A5}
\left\{
\begin{aligned}
&g''_{11}f_{4}+g''_{13}f_{2}+g''_{15}+g_{11}f''_{4}+g_{13}f''_{2}-2g'_{11}f'_{4}-2g'_{13}f'_{2}\\
&+(\mu+\lambda)(g_{11}f_{4}+g_{13}f_{2}+g_{15})+c_{1}(s_{2}g_{-13}+s_{4}g_{-11})=0,\\
&g''_{-11}f_{4}+g''_{-13}f_{2}+g''_{-15}+g_{-11}f''_{4}+g_{-13}f''_{2}-2g'_{-11}f'_{4}-2g'_{-13}f'_{2}\\
&+(\mu-\lambda)(g_{-11}f_{4}+g_{-13}f_{2}+g_{-15})+c_{1}(s_{2}g_{13}+s_{4}g_{11})=0,\\
&g''_{01}f_{4}+g''_{03}f_{2}+g''_{05}+g_{01}f''_{4}+g_{03}f''_{2}-2g'_{01}f'_{4}-2g'_{03}f'_{2}\\
&+\mu(g_{01}f_{4}+g_{03}f_{2}+g_{05})-c_{1}(s_{2}g_{03}+s_{4}g_{01})=0,
\end{aligned}
\right.
\end{equation}
\begin{equation}\label{A6}
\left\{
\begin{aligned}
&2f'_{2}f'_{4}-(f''_{2}f_{4}+f''_{4}f_{2}+f''_{6})+\frac{c_{0}+c_{1}}{2}(2g_{11}g_{15}+g^{2}_{13}+2g_{-11}g_{-15}+g^{2}_{-13}+2g_{01}g_{05}+g^{2}_{03})=0,\\
&s_{2}f_{4}+s_{4}f_{2}+s_{6}-(2g_{01}g_{05}+g^{2}_{03})+2(g_{11}g_{-15}+g_{13}g_{-13}+g_{15}g_{-11})=0,
\end{aligned}
\right.
\end{equation}
\begin{equation}\label{A7}
\left\{
\begin{aligned}
&g''_{11}f_{6}+g''_{13}f_{4}+g''_{15}f_{2}+g''_{17}+g_{11}f''_{6}+g_{13}f''_{4}+g_{15}f''_{2}-2(g'_{11}f'_{6}+g'_{13}f'_{4}+g'_{15}f'_{2})\\
&+(\mu+\lambda)(g_{11}f_{6}+g_{13}f_{4}+g_{15}f_{2}+g_{17})+c_{1}(s_{2}g_{-15}+s_{4}g_{-13}+s_{6}g_{-11})=0,\\
&g''_{-11}f_{6}+g''_{-13}f_{4}+g''_{-15}f_{2}+g''_{-17}+g_{-11}f''_{6}+g_{-13}f''_{4}+g_{-15}f''_{2}-2(g'_{-11}f'_{6}+g'_{-13}f'_{4}+g'_{-15}f'_{2})\\
&+(\mu-\lambda)(g_{-11}f_{6}+g_{-13}f_{4}+g_{-15}f_{2}+g_{-17})+c_{1}(s_{2}g_{15}+s_{4}g_{13}+s_{6}g_{11})=0,\\
&g''_{01}f_{6}+g''_{03}f_{4}+g''_{05}f_{2}+g''_{07}+g_{01}f''_{6}+g_{03}f''_{4}+g_{05}f''_{2}-2(g'_{01}f'_{6}+g'_{03}f'_{4}+g'_{05}f'_{2})\\
&+\mu(g_{01}f_{6}+g_{03}f_{4}+g_{05}f_{2}+g_{07})-c_{1}(s_{2}g_{05}+s_{4}g_{03}+s_{6}g_{01})=0,
\end{aligned}
\right.
\end{equation}
\begin{equation}\label{A8}
\left\{
\begin{aligned}
&2f'_{2}f'_{6}+(f'_{4})^{2}-(f''_{2}f_{6}+f''_{4}f_{4}+f''_{6}f_{2}+f''_{8})\\
&+\frac{c_{0}+c_{1}}{2}(2g_{11}g_{17}+2g_{13}g_{15}+2g_{-11}g_{-17}+2g_{-13}g_{-15}+2g_{01}g_{07}+2g_{03}g_{05})=0,\\
&s_{2}f_{6}+s_{4}f_{4}+s_{6}f_{2}-(2g_{01}g_{07}+2g_{03}g_{05})+2(g_{11}g_{-17}+g_{13}g_{-15}+g_{15}g_{-13}+g_{17}g_{-11})=0,
\end{aligned}
\right.
\end{equation}
\begin{equation}\label{A9}
\left\{
\begin{aligned}
&g''_{11}f_{8}+g''_{13}f_{6}+g''_{15}f_{4}+g''_{17}f_{2}+g_{11}f''_{8}+g_{13}f''_{6}+g_{15}f''_{4}+g_{17}f''_{2}-2(g'_{11}f'_{8}+g'_{13}f'_{6}+g'_{15}f'_{4}+g'_{17}f'_{2}) \\
&+(\mu+\lambda)(g_{11}f_{8}+g_{13}f_{6}+g_{15}f_{4}+g_{17}f_{2})+c_{1}(s_{2}g_{-17}+s_{4}g_{-15}+s_{6}g_{-13})=0,\\
&g''_{-11}f_{8}+g''_{-13}f_{6}+g''_{-15}f_{4}+g''_{-17}f_{2}+g_{-11}f''_{8}+g_{-13}f''_{6}+g_{-15}f''_{4}+g_{-17}f''_{2}\\
&-2(g'_{-11}f'_{8}+g'_{-13}f'_{6}+g'_{-15}f'_{4}+g'_{-17}f'_{2})+(\mu-\lambda)(g_{-11}f_{8}+g_{-13}f_{6}+g_{-15}f_{4}+g_{-17}f_{2}) \\
&+c_{1}(s_{2}g_{17}+s_{4}g_{15}+s_{6}g_{13})=0,\\
&g''_{01}f_{8}+g''_{03}f_{6}+g''_{05}f_{4}+g''_{07}f_{2}+g_{01}f''_{8}+g_{03}f''_{6}+g_{05}f''_{4}+g_{07}f''_{2}-2(g'_{01}f'_{8}+g'_{03}f'_{6}+g'_{05}f'_{4}+g'_{07}f'_{2}) \\
&+\mu(g_{01}f_{8}+g_{03}f_{6}+g_{05}f_{4}+g_{07}f_{2})-c_{1}(s_{2}g_{07}+s_{4}g_{05}+s_{6}g_{03})=0,
\end{aligned}
\right.
\end{equation}
\begin{equation}\label{A10}
\left\{
\begin{aligned}
&2f'_{2}f'_{8}+2f'_{4}f'_{6}-(f''_{2}f_{8}+f''_{4}f_{6}+f''_{6}f_{4}+f''_{8}f_{2})\\
&+\frac{c_{0}+c_{1}}{2}(2g_{13}g_{17}+g^{2}_{15}+2g_{-13}g_{-17}+g^{2}_{-15}+2g_{03}g_{07}+g^{2}_{05})=0,\\
&s_{2}f_{8}+s_{4}f_{6}+s_{6}f_{4}-(2g_{03}g_{07}+g^{2}_{05})+2(g_{13}g_{-17}+g_{15}g_{-15}+g_{17}g_{-13})=0,
\end{aligned}
\right.
\end{equation}
\begin{equation}\label{A11}
\left\{
\begin{aligned}
&g''_{13}f_{8}+g''_{15}f_{6}+g''_{17}f_{4}+g_{13}f''_{8}+g_{15}f''_{6}+g_{17}f''_{4}-2(g'_{13}f'_{8}+g'_{15}f'_{6}+g'_{17}f'_{4})\\
&+(\mu+\lambda)(g_{13}f_{8}+g_{15}f_{6}+g_{17}f_{4})+c_{1}(s_{4}g_{-17}+s_{6}g_{-15})=0,\\
&g''_{-13}f_{8}+g''_{-15}f_{6}+g''_{-17}f_{4}+g_{-13}f''_{8}+g_{-15}f''_{6}+g_{-17}f''_{4}-2(g'_{-13}f'_{8}+g'_{-15}f'_{6}+g'_{-17}f'_{4})\\
&+(\mu-\lambda)(g_{-13}f_{8}+g_{-15}f_{6}+g_{-17}f_{4})+c_{1}(s_{4}g_{17}+s_{6}g_{15})=0,\\
&g''_{03}f_{8}+g''_{05}f_{6}+g''_{07}f_{4}+g_{03}f''_{8}+g_{05}f''_{6}+g_{07}f''_{4}-2(g'_{03}f'_{8}+g'_{05}f'_{6}+g'_{07}f'_{4})\\
&+\mu(g_{03}f_{8}+g_{05}f_{6}+g_{07}f_{4})-c_{1}(s_{4}g_{07}+s_{6}g_{05})=0,
\end{aligned}
\right.
\end{equation}
\begin{equation}\label{A12}
\left\{
\begin{aligned}
&2f'_{4}f'_{8}+(f'_{6})^{2}-(f''_{4}f_{8}+f''_{6}f_{6}+f''_{8}f_{4})+\frac{c_{0}+c_{1}}{2}(2g_{15}g_{17}+2g_{-15}g_{-17}+2g_{05}g_{07})=0,\\
&s_{4}f_{8}+s_{6}f_{6}-2g_{05}g_{07}+2(g_{15}g_{-17}+g_{17}g_{-15})=0,
\end{aligned}
\right.
\end{equation}
\begin{equation}\label{A13}
\left\{
\begin{aligned}
&g''_{15}f_{8}+g''_{17}f_{6}+g_{15}f''_{8}+g_{17}f''_{6}-2(g'_{15}f'_{8}+g'_{17}f'_{6})+(\mu+\lambda)(g_{15}f_{8}+g_{17}f_{6})+c_{1}s_{6}g_{-17}=0,\\
&g''_{-15}f_{8}+g''_{-17}f_{6}+g_{-15}f''_{8}+g_{-17}f''_{6}-2(g'_{-15}f'_{8}+g'_{-17}f'_{6})+(\mu-\lambda)(g_{-15}f_{8}+g_{-17}f_{6})+c_{1}s_{6}g_{17}=0,\\
&g''_{05}f_{8}+g''_{07}f_{6}+g_{05}f''_{8}+g_{07}f''_{6}-2(g'_{05}f'_{8}+g'_{07}f'_{6})+\mu(g_{05}f_{8}+g_{07}f_{6})-c_{1}s_{6}g_{07}=0,\\
\end{aligned}
\right.
\end{equation}
\begin{equation}\label{A14}
\left\{
\begin{aligned}
&2f'_{6}f'_{8}-(f''_{6}f_{8}+f''_{8}f_{6})+\frac{c_{0}+c_{1}}{2}(g^{2}_{17}+g^{2}_{-17}+g^{2}_{07})=0,\\
&s_{6}f_{8}-g^{2}_{07}+2g_{17}g_{-17}=0,
\end{aligned}
\right.
\end{equation}
\begin{equation}\label{A15}
\left\{
\begin{aligned}
&g''_{17}f_{8}+g_{17}f''_{8}-2g'_{17}f'_{8}+(\mu+\lambda)g_{17}f_{8}=0,\\
&g''_{-17}f_{8}+g_{-17}f''_{8}-2g'_{-17}f'_{8}+(\mu-\lambda)g_{-17}f_{8}=0,\\
&g''_{07}f_{8}+g_{07}f''_{8}-2g'_{07}f'_{8}+\mu g_{07}f_{8}=0,
\end{aligned}
\right.
\end{equation}
\begin{equation}\label{A16}
(f'_{8})^{2}-f''_{8}f_{8}=0.
\end{equation}
The following lemma solves the formal expansions of \eqref{L1}.
\begin{lemma}\label{LE1}
The equations \eqref{A1}-\eqref{A16} are solved by
\begin{align}\label{C2}
g_{11}=a_{1}e^{\eta_{1}x},\ g_{-11}=a_{-1}e^{\eta_{-1}x},\ g_{01}=a_{0}e^{\eta_{0}x},
\end{align}
\begin{align}\label{C3}
f_{2}&=c\left(\frac{a^{2}_{1}}{4\eta^{2}_{1}}e^{2\eta_{1}x}+\frac{a^{2}_{-1}}{4\eta^{2}_{-1}}e^{2\eta_{-1}x}+\frac{a^{2}_{0}}{4\eta^{2}_{0}}e^{2\eta_{0}x}\right),\\
s_{2}&=a^{2}_{0}e^{2\eta_{0}x}-2a_{1}a_{-1}e^{(\eta_{1}+\eta_{-1})x},
\end{align}
\begin{align}\label{B1}
g_{13}&= \frac{ca_1 a_0^2 (\eta_1 - \eta_0)}{4\eta_0^2 (\eta_1 + \eta_0)} e^{(\eta_1 + 2\eta_0)x} - \frac{ca_0^2 a_{-1}}{2(\eta_0 + \eta_{-1})^2} e^{(2\eta_0 + \eta_{-1})x} + \frac{ca_1 a_{-1}^2}{4\eta_{-1}^2} e^{(\eta_1 + 2\eta_{-1})x},
\end{align}
\begin{align}\label{B2}
g_{-1,3} = \frac{ca_{-1}a_0^2(\eta_{-1} - \eta_0)}{4\eta_0^2(\eta_{-1} + \eta_0)} e^{(\eta_{-1} + 2\eta_0)x} - \frac{ca_0^2a_1}{2(\eta_0 + \eta_1)^2} e^{(2\eta_0 + \eta_1)x} + \frac{ca_{-1}a_1^2}{4\eta_1^2} e^{(\eta_{-1} + 2\eta_1)x},
\end{align}
\begin{align}\label{B3}
g_{0,3} &= \frac{ca_0 a_1^2 (\eta_0 - \eta_1)}{4\eta_1^2 (\eta_0 + \eta_1)} e^{(\eta_0 + 2\eta_1)x} + \frac{ca_0^3}{8\eta_0^2} e^{3\eta_0 x} - \frac{ca_0 a_1 a_{-1}}{(\eta_0 + \eta_1)(\eta_0 + \eta_{-1})} e^{(\eta_1 + \eta_0 + \eta_{-1})x} \\\nonumber
&+ \frac{ca_0 a_{-1}^2 (\eta_0 - \eta_{-1})}{4\eta_{-1}^2 (\eta_0 + \eta_{-1})} e^{(\eta_0 + 2\eta_{-1})x},
\end{align}
\begin{align}\label{B4}
f_4 =& \frac{c^2 a_0^4}{64\eta_0^4} e^{4\eta_0 x} + \frac{c^2 a_0^2 a_1^2 (\eta_0 - \eta_1)^2}{16\eta_0^2 \eta_1^2 (\eta_0 + \eta_1)^2} e^{2(\eta_0+\eta_1)x}
+ \frac{c^2 a_0^2 a_{-1}^2 (\eta_0 - \eta_{-1})^2}{16\eta_0^2 \eta_{-1}^2 (\eta_0 + \eta_{-1})^2} e^{2(\eta_0+\eta_{-1})x} \\\nonumber
&+ \frac{c^2 a_1^2 a_{-1}^2}{16\eta_1^2 \eta_{-1}^2} e^{2(\eta_1+\eta_{-1})x}
- \frac{c^2 a_0^2 a_1 a_{-1}}{(\eta_0 + \eta_1)^2 (\eta_0 + \eta_{-1})^2} e^{(2\eta_0+\eta_1+\eta_{-1})x},
\end{align}
\begin{align}\label{C1}
s_{4}&= \frac{c a_0^2 a_1^2 (\eta_1 - \eta_0)^2}{4\eta_1^2 (\eta_0 + \eta_1)^2} e^{2(\eta_0+\eta_1)x} + \frac{c a_0^2 a_{-1}^2 (\eta_{-1} - \eta_0)^2}{4\eta_{-1}^2 (\eta_0 + \eta_{-1})^2} e^{2(\eta_0+\eta_{-1})x}+ \frac{c a_0^2 a_1 a_{-1} (\eta_1 - \eta_{-1})^2}{2\eta_0^2 (2\eta_0 + \eta_1 + \eta_{-1})^2} e^{(2\eta_0+\eta_1+\eta_{-1})x},
\end{align}
\begin{align}\label{C4}
g_{1,5} = & \frac{a_0^4 a_1 c^2 (\eta_0 - \eta_1)^2 (2\eta_0 - \eta_1 - \eta_{-1})^2 e^{4\eta_0 x + \eta_1 x}}{64\eta_0^4 (\eta_0 - \eta_{-1})^2 (\eta_1 - \eta_{-1})^2} \\\nonumber
& - \frac{a_0^2 a_1^2 a_{-1} c^2 (2\eta_0 - \eta_1 - \eta_{-1})^2 (2\eta_0^2 - 2\eta_0\eta_1 - 2\eta_0\eta_{-1} + \eta_1^2 + \eta_{-1}^2)^4 e^{2\eta_0 x + 2\eta_1 x + \eta_{-1} x}}{32\eta_0^4 (\eta_0 - \eta_{-1})^2 (\eta_1 - \eta_{-1})^4 (2\eta_0^2 - 4\eta_0\eta_1 + \eta_1^2 + 2\eta_1\eta_{-1} - \eta_{-1}^2)^2} \\\nonumber
& - \frac{a_0^2 a_1 a_{-1}^2 c^2 (\eta_0 - \eta_{-1}) (2\eta_0 - \eta_1 - \eta_{-1})^3 (2\eta_0^2 - 2\eta_0\eta_1 - 2\eta_0\eta_{-1} + \eta_1^2 + \eta_{-1}^2)^2 e^{2\eta_0 x + \eta_1 x + 2\eta_{-1} x}}{16\eta_0^4 (\eta_0 - \eta_1) (\eta_1 - \eta_{-1})^3 (2\eta_0^2 - 4\eta_0\eta_{-1} - \eta_1^2 + 2\eta_1\eta_{-1} + \eta_{-1}^2)^2},
\end{align}
 \begin{align}\label{C5}
g_{-1,5} = & \frac{a_0^4 a_{-1} c^2 (\eta_0 - \eta_{-1})^2 (2\eta_0 - \eta_1 - \eta_{-1})^2 e^{4\eta_0 x + \eta_{-1} x}}{64\eta_0^4 (\eta_0 - \eta_1)^2 (\eta_{-1} - \eta_1)^2} \\\nonumber
& + \frac{a_0^2 a_1^2 a_{-1} c^2 (\eta_0 - \eta_1) (2\eta_0 - \eta_1 - \eta_{-1})^3 (2\eta_0^2 - 2\eta_0\eta_1 - 2\eta_0\eta_{-1} + \eta_1^2 + \eta_{-1}^2)^2 e^{2\eta_0 x + 2\eta_1 x + \eta_{-1} x}}{16\eta_0^4 (\eta_0 - \eta_{-1}) (\eta_1 - \eta_{-1})^3 (2\eta_0^2 - 4\eta_0\eta_1 + \eta_1^2 + 2\eta_1\eta_{-1} - \eta_{-1}^2)^2} \\\nonumber
& - \frac{a_0^2 a_1 a_{-1}^2 c^2 (2\eta_0 - \eta_1 - \eta_{-1})^2 (2\eta_0^2 - 2\eta_0\eta_1 - 2\eta_0\eta_{-1} + \eta_1^2 + \eta_{-1}^2)^4 e^{2\eta_0 x + \eta_1 x + 2\eta_{-1} x}}{32\eta_0^4 (\eta_0 - \eta_1)^2 (\eta_1 - \eta_{-1})^4 (2\eta_0^2 - 4\eta_0\eta_{-1} - \eta_1^2 + 2\eta_1\eta_{-1} + \eta_{-1}^2)^2},
\end{align}
\begin{align}\label{C6}
g_{0,5} = & \frac{a_0^3 a_1^2 c^2 (\eta_0 - \eta_1)^3 (2\eta_0 - \eta_1 - \eta_{-1})^3 (2\eta_0^2 - 2\eta_0\eta_1 - 2\eta_0\eta_{-1} + \eta_1^2 + \eta_{-1}^2)^2 e^{3\eta_0 x + 2\eta_1 x}}{32\eta_0^4 (\eta_0 - \eta_{-1})^3 (\eta_1 - \eta_{-1})^3 (2\eta_0^2 - 4\eta_0\eta_1 + \eta_1^2 + 2\eta_1\eta_{-1} - \eta_{-1}^2)^2} \\\nonumber
& - \frac{a_0^3 a_1 a_{-1} c^2 (2\eta_0 - \eta_1 - \eta_{-1})^2 (2\eta_0^2 - 2\eta_0\eta_1 - 2\eta_0\eta_{-1} + \eta_1^2 + \eta_{-1}^2) e^{3\eta_0 x + \eta_1 x + \eta_{-1} x}}{32\eta_0^4 (\eta_0 - \eta_1)(\eta_0 - \eta_{-1})(\eta_1 - \eta_{-1})^4} \\\nonumber
& - \frac{a_0^3 a_{-1}^2 c^2 (\eta_0 - \eta_{-1})^3 (2\eta_0 - \eta_1 - \eta_{-1})^3 (2\eta_0^2 - 2\eta_0\eta_1 - 2\eta_0\eta_{-1} + \eta_1^2 + \eta_{-1}^2)^2 e^{3\eta_0 x + 2\eta_{-1} x}}{32\eta_0^4 (\eta_0 - \eta_1)^3 (\eta_1 - \eta_{-1})^3 (2\eta_0^2 - 4\eta_0\eta_{-1} - \eta_1^2 + 2\eta_1\eta_{-1} + \eta_{-1}^2)^2} \\\nonumber
& - \frac{a_0 a_1^2 a_{-1}^2 c^2 (2\eta_0 - \eta_1 - \eta_{-1})^2 (2\eta_0^2 - 2\eta_0\eta_1 - 2\eta_0\eta_{-1} + \eta_1^2 + \eta_{-1}^2)^4 e^{\eta_0 x + 2\eta_1 x + 2\eta_{-1} x}}{16\eta_0^4 (\eta_1 - \eta_{-1})^2 (2\eta_0^2 - 4\eta_0\eta_1 + \eta_1^2 + 2\eta_1\eta_{-1} - \eta_{-1}^2)^2 (2\eta_0^2 - 4\eta_0\eta_{-1} - \eta_1^2 + 2\eta_1\eta_{-1} + \eta_{-1}^2)^2},
\end{align}
\begin{align}\label{C7}
f_6 = & \frac{a_0^4 a_1^2 c^3 (\eta_0 - \eta_1)^4 (2\eta_0^2 - 2\eta_0\eta_1 - 2\eta_0\eta_{-1} + \eta_1^2 + \eta_{-1}^2)^2 (2\eta_0 - \eta_1 - \eta_{-1})^4 e^{4\eta_0 x + 2\eta_1 x}}{256\eta_0^6 (\eta_0 - \eta_{-1})^4 (\eta_1 - \eta_{-1})^4 (2\eta_0^2 - 4\eta_0\eta_1 + \eta_1^2 + 2\eta_1\eta_{-1} - \eta_{-1}^2)^2} \\\nonumber
& + \frac{a_0^4 a_{-1}^2 c^3 (\eta_0 - \eta_{-1})^4 (2\eta_0^2 - 2\eta_0\eta_1 - 2\eta_0\eta_{-1} + \eta_1^2 + \eta_{-1}^2)^2 (2\eta_0 - \eta_1 - \eta_{-1})^4 e^{4\eta_0 x + 2\eta_{-1} x}}{256\eta_0^6 (\eta_0 - \eta_1)^4 (\eta_1 - \eta_{-1})^4 (2\eta_0^2 - 4\eta_0\eta_{-1} - \eta_1^2 + 2\eta_1\eta_{-1} + \eta_{-1}^2)^2} \\\nonumber
& + \frac{a_0^2 a_1^2 a_{-1}^2 c^3 (2\eta_0^2 - 2\eta_0\eta_1 - 2\eta_0\eta_{-1} + \eta_1^2 + \eta_{-1}^2)^4 (2\eta_0 - \eta_1 - \eta_{-1})^4 e^{2\eta_0 x + 2\eta_1 x + 2\eta_{-1} x}}{64\eta_0^6 (\eta_1 - \eta_{-1})^4 (2\eta_0^2 - 4\eta_0\eta_1 + \eta_1^2 + 2\eta_1\eta_{-1} - \eta_{-1}^2)^2 (2\eta_0^2 - 4\eta_0\eta_{-1} - \eta_1^2 + 2\eta_1\eta_{-1} + \eta_{-1}^2)^2},
\end{align}
\begin{align}\label{C8}
s_6 = & \frac{a_0^2 a_1^2 a_{-1}^2 c^2 (2\eta_0 - \eta_1 - \eta_{-1})^4 (2\eta_0^2 - 2\eta_0\eta_1 - 2\eta_0\eta_{-1} + \eta_1^2 + \eta_{-1}^2)^4 e^{2\eta_0 x + 2\eta_1 x + 2\eta_{-1} x}}{16\eta_0^4 (\eta_1 - \eta_{-1})^4 (2\eta_0^2 - 4\eta_0\eta_1 + \eta_1^2 + 2\eta_1\eta_{-1} - \eta_{-1}^2)^2 (2\eta_0^2 - 4\eta_0\eta_{-1} - \eta_1^2 + 2\eta_1\eta_{-1} + \eta_{-1}^2)^2} \\\nonumber
& - \frac{a_0^4 a_1 a_{-1} c^2 (2\eta_0 - \eta_1 - \eta_{-1})^4 e^{4\eta_0 x + \eta_1 x + \eta_{-1} x}}{32\eta_0^4 (\eta_1 - \eta_{-1})^4},
\end{align}
\begin{align}\label{C9}
g_{1,7} = \frac{a_0^4 a_1 a_{-1}^2 c^3 (\eta_0 - \eta_{-1})^2 (2\eta_0 - \eta_1 - \eta_{-1})^6 (2\eta_0^2 - 2\eta_0\eta_1 - 2\eta_0\eta_{-1} + \eta_1^2 + \eta_{-1}^2)^2 e^{4\eta_0 x + \eta_1 x + 2\eta_{-1} x}}{256\eta_0^6 (\eta_0 - \eta_1)^2 (\eta_1 - \eta_{-1})^6 (2\eta_0^2 - 4\eta_0\eta_{-1} - \eta_1^2 + 2\eta_1\eta_{-1} + \eta_{-1}^2)^2},
\end{align}
\begin{align}
g_{-1,7} = \frac{a_0^4 a_1^2 a_{-1} c^3 (\eta_0 - \eta_1)^2 (2\eta_0 - \eta_1 - \eta_{-1})^6 (2\eta_0^2 - 2\eta_0\eta_1 - 2\eta_0\eta_{-1} + \eta_1^2 + \eta_{-1}^2)^2 e^{4\eta_0 x + 2\eta_1 x + \eta_{-1} x}}{256\eta_0^6 (\eta_0 - \eta_{-1})^2 (\eta_1 - \eta_{-1})^6 (2\eta_0^2 - 4\eta_0\eta_1 + \eta_1^2 + 2\eta_1\eta_{-1} - \eta_{-1}^2)^2}
\end{align}
\begin{align}\label{C10}
g_{0,7} = - \frac{a_0^3 a_1^2 a_{-1}^2 c^3 (2\eta_0 - \eta_1 - \eta_{-1})^6 (2\eta_0^2 - 2\eta_0\eta_1 - 2\eta_0\eta_{-1} + \eta_1^2 + \eta_{-1}^2)^4 e^{3\eta_0 x + 2\eta_1 x + 2\eta_{-1} x}}{128\eta_0^6 (\eta_1 - \eta_{-1})^6 (2\eta_0^2 - 4\eta_0\eta_1 + \eta_1^2 + 2\eta_1\eta_{-1} - \eta_{-1}^2)^2 (2\eta_0^2 - 4\eta_0\eta_{-1} - \eta_1^2 + 2\eta_1\eta_{-1} + \eta_{-1}^2)^2},
\end{align}
\begin{align}\label{C11}
f_8 = \frac{a_0^4 a_1^2 a_{-1}^2 c^4 (2\eta_0 - \eta_1 - \eta_{-1})^8 (2\eta_0^2 - 2\eta_0\eta_1 - 2\eta_0\eta_{-1} + \eta_1^2 + \eta_{-1}^2)^4 e^{4\eta_0 x + 2\eta_1 x + 2\eta_{-1} x}}{1024\eta_0^8 (\eta_1 - \eta_{-1})^8 (2\eta_0^2 - 4\eta_0\eta_1 + \eta_1^2 + 2\eta_1\eta_{-1} - \eta_{-1}^2)^2 (2\eta_0^2 - 4\eta_0\eta_{-1} - \eta_1^2 + 2\eta_1\eta_{-1} + \eta_{-1}^2)^2}.
\end{align}
\end{lemma}
\begin{proof}
Let $$\eta_{1}=\sqrt{|(\mu+\lambda)|},\ \eta_{-1}=\sqrt{|(\mu-\lambda)|},\ \eta_{0}=\sqrt{|\mu|},$$we have  $$\eta^{2}_{1}+\eta^{2}_{-1}=2\eta^{2}_{0}.$$By \eqref{A1}, we get

$$g_{11}=a_{1}e^{\eta_{1}x},\ g_{-11}=a_{-1}e^{\eta_{-1}x},\ g_{01}=a_{0}e^{\eta_{0}x}.$$ Since $$f''_{2}=c(g^{2}_{11}+g^{2}_{-11}+g^{2}_{01})=c(a^{2}_{1}e^{2\eta_{1}x}+a^{2}_{-1}e^{2\eta_{-1}x}+a^{2}_{0}e^{2\eta_{0}x}).$$
From \eqref{A2}, we have $$f_{2}=c\left(\frac{a^{2}_{1}}{4\eta^{2}_{1}}e^{2\eta_{1}x}+\frac{a^{2}_{-1}}{4\eta^{2}_{-1}}e^{2\eta_{-1}x}+\frac{a^{2}_{0}}{4\eta^{2}_{0}}e^{2\eta_{0}x}\right),$$
and $$s_{2}=a^{2}_{0}e^{2\eta_{0}x}-2a_{1}a_{-1}e^{(\eta_{1}+\eta_{-1})x}.$$ \eqref{B1}-\eqref{B3} can be obtained by \eqref{A3}.
Since
$$f''_{4}=(f'_{2})^{2}-f_{2}f''_{2}+(c_{0}+c_{1})(g_{11}g_{13}+g_{-11}g_{-13}+g_{01}g_{03}),$$
and
$$s_{4}=-s_{2}f_{2}+2g_{01}g_{03}-2g_{11}g_{-13}-2g_{13}g_{-11}.$$
Through a large number of complex calculations, we get \eqref{B4} and \eqref{C1}. By \eqref{A5}, we have
$$g''_{15}+(\mu+\lambda)g_{15}=-(g''_{13}+(\mu+\lambda)g_{13})f_{2}-g_{11}f''_{4}-g_{13}f''_{2}+2g'_{11}f'_{4}+2g'_{13}f'_{2}-c_{1}(s_{2}g_{-13}+s_{4}g_{-11}),$$
through extremely complex calculations and solving the second-order ODE, we can obtain \eqref{C4} - \eqref{C6}. By \eqref{A6}, we have \eqref{C7} and \eqref{C8}. By \eqref{A7}, we have \eqref{C9}-\eqref{C10}. Finally, \eqref{C11} can be obtained by \eqref{A8} and some complex calculations.
\end{proof}

\section{Cascade functions
and determinant representation}\label{S2}
Using the matrices defined in \eqref{M1},\eqref{M2},\eqref{M3} and \eqref{T2}, let
\begin{equation}\label{M4}
f = \det M, \qquad G = R \, \text{adj}(M) A R^T, \qquad G = \begin{pmatrix} g_1 & g_0 / \sqrt{2} \\ g_0 / \sqrt{2} & g_{-1} \end{pmatrix},
\end{equation}
where \(\text{adj}(M)\) denotes the adjugate matrix, i.e., the transpose of the cofactor matrix, satisfying \(M \, \text{adj}(M) = (\det M) I_4\). It is neither the transpose matrix nor the conjugate transpose. We will prove later that for \(f > 0\), \(G/f\) is exactly \eqref{T1}.

For any \(4 \times 4\) matrix \(M\) and 4-dimensional column vectors \(v, w\), first take the Schur complement, where \(M\) is invertible:
\begin{equation}\label{M5}
\begin{aligned}
\det \begin{pmatrix} M & v \\ w^T & 0 \end{pmatrix} &= \det M \det(-w^T M^{-1} v) = -(\det M)w^T M^{-1} v = -w^T \text{adj}(M) v.
\end{aligned}
\end{equation}
Let
\begin{equation}\label{M6}
r_1 = (1, 1, 0, 0)^T, \qquad r_{-1} = (0, 0, 1, 1)^T, \qquad R = \begin{pmatrix} r_1^T \\ r_{-1}^T \end{pmatrix}.
\end{equation}

\noindent Reading term by term from \eqref{M4}, we obtain
\begin{equation}\label{M7}
\begin{aligned}
g_1 &= r_1^T \, \text{adj}(M) A r_1, \ \
\frac{g_0}{\sqrt{2}} = r_1^T \, \text{adj}(M) A r_{-1}, \ \
g_{-1} = r_{-1}^T \, \text{adj}(M) A r_{-1}.
\end{aligned}
\end{equation}
\noindent Both sides are polynomials in the elements of \(M, v, w\), so \eqref{M5} also holds for singular \(M\). Substituting the three sets of vectors from \eqref{M7} into \eqref{M5}, we obtain
\begin{equation}\label{M8}
g_1 = -\det \begin{pmatrix} M & \mathcal{A}r_1 \\ r_1^T & 0 \end{pmatrix}, \qquad g_{-1} = -\det \begin{pmatrix} M & \mathcal{A}r_{-1} \\ r_{-1}^T & 0 \end{pmatrix},\ g_0 = -\sqrt{2} \det \begin{pmatrix} M & \mathcal{A}r_{-1} \\ r_1^T & 0 \end{pmatrix}.
\end{equation}
The negative sign comes from \eqref{M5}, and the \(\sqrt{2}\) before \(g_0\) comes from \(G_{12} = g_0 / \sqrt{2}\). Redefine the six-order bordered matrix
\begin{equation*}
\mathcal{B} = \begin{pmatrix} M & \mathcal{A}R^T \\ R & 0_{2 \times 2} \end{pmatrix}, \qquad s := -2 \det \mathcal{B}.
\end{equation*}
Where \(M\) is invertible, the Schur complement step-by-step gives
\begin{equation}\label{M10}
\begin{aligned}
\det \mathcal{B} &= \det M \det(-RM^{-1}\mathcal{A}R^T)= f \det(-U_\mathbf{a}) = f \det U_\mathbf{a},
\end{aligned}
\end{equation}
where the last step uses \(U_\mathbf{a} \in \mathbb{R}^{2 \times 2}\), so \(\det(-U_\mathbf{a}) = (-1)^2 \det U_\mathbf{a}\). Hence
\begin{equation}\label{M11}
s = -2f \det U_\mathbf{a}.
\end{equation}
On the other hand, \(\text{adj}(M) = fM^{-1}\) gives
\begin{equation*}
G = fU_\mathbf{a}, \quad \det G = f^2 \det U_\mathbf{a}.
\end{equation*}
Multiplying \eqref{M11} by \(f\), and then directly expanding the second-order determinant:
\begin{equation}\label{M12}
\begin{aligned}
sf = -2f^2 \det U_\mathbf{a} = -2 \det G = -2 \det \begin{pmatrix} g_1 & g_0 / \sqrt{2} \\ g_0 / \sqrt{2} & g_{-1} \end{pmatrix} = g_0^2 - 2g_1 g_{-1}.
\end{aligned}
\end{equation}
Equations \eqref{M10}--\eqref{M12} hold first where \(M\) is invertible; since both sides are polynomials in the matrix elements, and as \(x \to -\infty\), \(M \to I_4\), they are also global algebraic identities.
\begin{lemma}\label{LE2}
 For any real amplitude \(a\) and any \(x \in \mathbb{R}\), the matrix \(C\) is positive definite,  \(M(x)\) is invertible and \(f(x) = \det M(x) \ge 1\).
\end{lemma}
\begin{proof}
For positive numbers \(p \neq q\), we have
\begin{equation*}
\det \begin{pmatrix} (2p)^{-1} & (p+q)^{-1} \\ (p+q)^{-1} & (2q)^{-1} \end{pmatrix} = \frac{(p-q)^2}{4pq(p+q)^2} > 0.
\end{equation*}
The diagonal elements of the two blocks are positive, so both \(2 \times 2\) blocks in \eqref{M1} are positive definite. We can also define
\begin{equation*}
\Phi(t) = \begin{pmatrix} e^{\eta_1 t} & e^{\eta_0 t} & 0 & 0 \\ 0 & 0 & e^{\eta_0 t} & e^{\eta_{-1} t} \end{pmatrix}, \qquad C = \int_{-\infty}^{0} \Phi(t)^T \Phi(t) \, dt.
\end{equation*}
Since the two indices in the same row are different, its four columns are linearly independent as functions in \(L^2((-\infty, 0); \mathbb{R}^2)\). Here, it is not meant that there exist four linearly independent constant vectors in \(\mathbb{R}^2\).
Let \(H = C^{1/2} \mathcal{A} C^{1/2}\).  Then \(H\) is a real symmetric matrix, and
\begin{equation*}
C^{1/2} M C^{-1/2} = I_4 + cH^2.
\end{equation*}
\noindent The right-hand side is positive definite, with eigenvalues \(1 + c\rho_j^2 \ge 1\), where \(\rho_j\) are the real eigenvalues of \(H\). Hence
\begin{equation*}
f = \det(I_4 + cH^2) = \prod_{j=1}^{4} (1 + c\rho_j^2) \ge 1.
\end{equation*}
\noindent This does not require \(H\) or \(\mathcal{A}\) to be invertible. Note that \(M\) is generally not symmetric, so \(M\) cannot be directly called a positive definite matrix; what is used here is its similarity relation to a positive definite symmetric matrix.
\end{proof}
 Let
\begin{equation}\label{M13}
\eta_1 = \sqrt{-(\mu + \lambda)}, \qquad \eta_0 = \sqrt{-\mu}, \qquad \eta_{-1} = \sqrt{-(\mu - \lambda)}, \qquad \eta_* := \min_{i=1,0,-1} \eta_i > 0.
\end{equation}

\noindent Then the three \(\eta_i\) are pairwise distinct, and
\begin{equation}\label{M14}
\eta_1^2 + \eta_{-1}^2 = 2\eta_0^2, \qquad \eta_1 + \eta_{-1} \le 2\eta_0, \qquad 2\eta_0 > \eta_1, \eta_{-1}.
\end{equation}
To utilize the left asymptotic lemma, we provide the following left-end asymptotics.
\begin{lemma}\label{LE3}
For arbitrary \(a \in \mathbb{R}^3\), we have
\begin{equation*}
\begin{aligned}
u_i^a(x) &= a_i e^{\eta_i x} + O(e^{(\eta_i + 2\eta_*) x}), \\
(u_i^a)'(x) &= \eta_i a_i e^{\eta_i x} + O(e^{(\eta_i + 2\eta_*) x}),\  i=1,0,-1.
\end{aligned}
\end{equation*}
\end{lemma}
\begin{proof}
From Lemma \ref{LE1}, we can rewrite $f$ and $g_{i}(x)$ as following
\begin{equation}\label{M15}
f(x) = 1 + \sum_{\nu \in \mathcal{F}} b_{\nu} e^{\beta_{\nu} x}, \qquad g_i(x) = a_i e^{\eta_i x} + \sum_{\nu \in \mathcal{C}_i} d_{i,\nu} e^{\gamma_{i,\nu} x}.
\end{equation}
Since the exponent rate of \(f_2\) is not less than \(2\eta_*\); each monomial of \(f_4, f_6, f_8\) contains four, six, and eight positive exponential factors, respectively, thus
\begin{equation}\label{M20}
\beta_{\nu} \geq 2\eta_{*} \qquad (\nu \in \mathcal{F}).
\end{equation}
The ordinary tail terms of the numerators have at least two additional positive exponential factors compared to \(e^{\eta_i x}\). For Spin-exchange terms, using \eqref{M14}, we have \begin{equation}\label{M16}
2\eta_0 + \eta_{-1} - \eta_1 \geq 2\eta_{-1} \geq 2\eta_*, \qquad 2\eta_0 + \eta_1 - \eta_{-1} \geq 2\eta_1 \geq 2\eta_*.
\end{equation}
The exchange terms of the middle component satisfy
\[
(\eta_1 + \eta_0 + \eta_{-1}) - \eta_0 = \eta_1 + \eta_{-1} \geq 2\eta_*.
\]
Applying the above three types of inequalities term by term, we obtain
\begin{equation}\label{M17}
\gamma_{i,\nu} \geq \eta_i + 2\eta_* \quad (\nu \in \mathcal{C}_i, \ i = 1, 0, -1).
\end{equation}

\noindent Since the sum in \eqref{M15} is finite, \eqref{M20}--\eqref{M17} yield
\begin{equation}\label{M18}
|f - 1| + |f'| \leq C e^{2\eta_* x},\ |g_i - a_i e^{\eta_i x}| + |g_i' - \eta_i a_i e^{\eta_i x}| \leq C e^{(\eta_i + 2\eta_*) x}, \ (x \leq x_0).
\end{equation}
From \(f \ge 1\) in Lemma \ref{LE2}, we have
\[
\frac{1}{f} - 1 = -\frac{f - 1}{f} = O(e^{2\eta_* x}), \qquad \left(\frac{1}{f}\right)' = -\frac{f'}{f^2} = O(e^{2\eta_* x}).
\]
Thus
\[
u_i^a - a_i e^{\eta_i x} = \frac{g_i - a_i e^{\eta_i x}}{f} + a_i e^{\eta_i x} \left(\frac{1}{f} - 1\right),
\]
\[
(u_i^a)' - \eta_i a_i e^{\eta_i x} = \frac{g_i' - \eta_i a_i e^{\eta_i x}}{f} + \eta_i a_i e^{\eta_i x} \left(\frac{1}{f} - 1\right) + g_i \left(\frac{1}{f}\right)'.
\]

\noindent By \eqref{M18}, we obtain
\begin{equation*}
\begin{aligned}
u_i^a(x) &= a_i e^{\eta_i x} + O(e^{(\eta_i + 2\eta_*) x}), \\
(u_i^a)'(x) &= \eta_i a_i e^{\eta_i x} + O(e^{(\eta_i + 2\eta_*) x}),i=1,0,-1.
\end{aligned}
\end{equation*}
This formula allows \(a_i = 0\), in which case the corresponding component may start from a higher order.
\end{proof}
Next, we consider the right-end decay, we give following lemma.
\begin{lemma}\label{LE4}
For arbitrary \(a \in \mathbb{R}^3\), we have
\begin{equation*}
U_\mathbf{a}(x) = O(e^{-\eta_* x}), \ \ U_\mathbf{a}'(x) = O(e^{-\eta_* x}) \quad (x \to +\infty).
\end{equation*}
\end{lemma}
\begin{proof}
Let \(E\) is a fixed non-zero coordinate subspace of \(\mathcal{A}\). If \(E\) is empty, then \(\mathcal{A} = 0\) and \(U_\mathbf{a} = 0\), and the conclusion is obvious. In the following, we assume that \(E\) is non-empty. Arranging the \(E\) coordinates in front, we have
\[
\mathcal{A} = \begin{pmatrix} \mathcal{A}_E & 0 \\ 0 & 0 \end{pmatrix}, \qquad \mathcal{C} = \begin{pmatrix} C_E & C_{E, E^\perp} \\ C_{E^\perp, E} & C_{E^\perp} \end{pmatrix}.
\]

\noindent Here \(\mathcal{A}_E\) is invertible, and \(C_E\) is a positive definite principal submatrix. Direct multiplication shows that \(M\) is a block upper triangular matrix, whose upper-left and lower-right blocks are respectively
\[
I_E + c\mathcal{A}_E C_E \mathcal{A}_E C_E, \quad I_{E^\perp}.
\]
Therefore, using \(R_E\) to denote the columns of \(R\) retaining \(E\), strictly speaking we have
\begin{equation*}
U_\mathbf{a} = R_E (I_E + c\mathcal{A}_E C_E \mathcal{A}_E C_E)^{-1} \mathcal{A}_E R_E^T.
\end{equation*}
Let \(H_E = C_E^{1/2} \mathcal{A}_E C_E^{1/2}\), \(L_E = H_E^{-1} = C_E^{-1/2} \mathcal{A}_E^{-1} C_E^{-1/2}\). By the matrix identity corresponding to the scalar rational function, we have
\begin{equation}\label{M22}
U_\mathbf{a} = R_E C_E^{-1/2} L_E (cI_E + L_E^2)^{-1} C_E^{-1/2} R_E^T.
\end{equation}
The non-zero elements of \(\mathcal{A}_E^{-1}\) are fixed coefficients multiplied by \(e^{-\eta_i x}\), hence
\[
L_E = O(e^{-\eta_* x}), \quad L_E' = O(e^{-\eta_* x}) \quad (x \to +\infty).
\]
 Write \(B_E = (cI_E + L_E^2)^{-1}\), then \(\|B_E\| \le c^{-1}\), and
\[
B_E' = -B_E(L_E'L_E + L_E L_E')B_E = O(e^{-2\eta_* x}).
\]
Here the order of multiplication is preserved, and no assumption is made that \(L_E\) commutes with its derivative. Substituting into \eqref{M22} yields
\begin{equation*}
U_\mathbf{a}(x) = O(e^{-\eta_* x}), \quad U_\mathbf{a}'(x) = O(e^{-\eta_* x}) \quad (x \to +\infty).
\end{equation*}
\end{proof}
\begin{lemma}\label{LE5}
Let
\begin{equation}
g_i = g_{i,1} + g_{i,3} + g_{i,5} + g_{i,7}, \quad f = 1 + f_2 + f_4 + f_6 + f_8, \quad s = s_2 + s_4 + s_6, i=1,0,-1,
\end{equation}
where $g_{i,1},g_{i,3},g_{i,5},g_{i,7},f_2,f_4,f_6,f_8,s_2,s_4,s_6$ are given in Lemma \ref{LE1}. Then they are identical to \eqref{M4}--\eqref{M12}, and satisfy five bilinear equations \eqref{eq9.1d} for any real \(\mathbf{a}\).
\end{lemma}
\begin{proof}
Set
\begin{equation*}
 z_1=e^{\eta_1x},\
 z_0=e^{\eta_0x},\
 z_{-1}=e^{\eta_{-1}x},
\end{equation*}
and write
\begin{equation}
 p=a_1z_1,\
 q=\frac{a_0z_0}{\sqrt{2}},\
 r=a_{-1}z_{-1}.
 \label{eq:app-en-pqr}
\end{equation}
For the five distinct entries of~$C$, introduce
\begin{equation*}
 \alpha=\frac{1}{2\eta_1},\
 \beta=\frac{1}{\eta_1+\eta_0},\
 \gamma=\frac{1}{2\eta_0},\
 \delta=\frac{1}{\eta_0+\eta_{-1}},\
 \varepsilon=\frac{1}{2\eta_{-1}}.
\end{equation*}
Then
\begin{equation}
 C=
 \begin{pmatrix}
  \alpha&\beta&0&0\\
  \beta&\gamma&0&0\\
  0&0&\gamma&\delta\\
  0&0&\delta&\varepsilon
 \end{pmatrix},
 \
 A=
 \begin{pmatrix}
  p&0&0&0\\
  0&0&q&0\\
  0&q&0&0\\
  0&0&0&r
 \end{pmatrix}.
 \label{eq:app-en-C-A}
\end{equation}
Define the determinants of the two second-order Gram blocks by
\begin{equation*}
 d_+:=\alpha\gamma-\beta^2,
 \
 d_-:=\gamma\varepsilon-\delta^2.
\end{equation*}
A direct reduction to common denominators gives
\begin{equation}
 d_+=
 \frac{(\eta_0-\eta_1)^2}
 {4\eta_1\eta_0(\eta_1+\eta_0)^2},
 \
 d_-=
 \frac{(\eta_0-\eta_{-1})^2}
 {4\eta_0\eta_{-1}(\eta_0+\eta_{-1})^2}.
 \label{eq:app-en-dpm-explicit}
\end{equation}
To compare homogeneous degrees in the amplitudes, replace~$A$ by~$XA$ and
define
\begin{equation*}
 M(X)=I_4+cX^2ACAC,
 \
 \widetilde f(X)=\det M(X),
\end{equation*}
together with
\begin{equation}
 \widetilde G(X)
 =R\operatorname{adj}(M(X))(XA)R^T
 =
 \begin{pmatrix}
  \widetilde g_1(X)&\widetilde g_0(X)/\sqrt{2}\\
  \widetilde g_0(X)/\sqrt{2}&\widetilde g_{-1}(X)
 \end{pmatrix}.
 \label{eq:app-en-scaled-G}
\end{equation}
Also set
\begin{equation}
 \widetilde B(X)=
 \begin{pmatrix}
  M(X)&XA R^T\\
  R&0_{2\times2}
 \end{pmatrix},
 \
 \widetilde s(X)=-2\det\widetilde B(X).
 \label{eq:app-en-scaled-B-s}
\end{equation}
We shall prove
\begin{equation*}
 \begin{aligned}
 \widetilde f(X)
 &=1+f_2X^2+f_4X^4+f_6X^6+f_8X^8,\\
 \widetilde g_i(X)
 &=g_{i,1}X+g_{i,3}X^3+g_{i,5}X^5+g_{i,7}X^7,
 \ i=1,0,-1,\\
 \widetilde s(X)
 &=s_2X^2+s_4X^4+s_6X^6,
 \end{aligned}
\end{equation*}
where the coefficients on the right-hand side are precisely the explicit
cascade functions listed in Lemma \ref{LE1}. Let $P:=AC.$
\eqref{eq:app-en-C-A} gives
\begin{equation}
 P=
 \begin{pmatrix}
  \alpha p&\beta p&0&0\\
  0&0&\gamma q&\delta q\\
  \beta q&\gamma q&0&0\\
  0&0&\delta r&\varepsilon r
 \end{pmatrix},
 \
 M(X)=I_4+cX^2P^2.
 \label{eq:app-en-P-explicit}
\end{equation}
Expanding the fourth-order determinant according to the sparse structure in
\eqref{eq:app-en-P-explicit}, we obtain
\begin{equation}
 \det(\lambda I_4-P)
 =\lambda^4-\sigma_1\lambda^3+\sigma_2\lambda^2
 -\sigma_3\lambda+\sigma_4,
 \label{eq:app-en-charpoly-P}
\end{equation}
where
\begin{equation}
 \begin{aligned}
 \sigma_1=\alpha p+\varepsilon r,\ \sigma_2=\alpha\varepsilon pr-\gamma^2q^2,\ \sigma_3=-\gamma q^2(pd_++rd_-),\ \sigma_4=-pq^2r\,d_+d_-.
 \end{aligned}
 \label{eq:app-en-sigma}
\end{equation}
Indeed, the coefficient of~$\lambda$ in
\eqref{eq:app-en-charpoly-P} is
$
 \gamma q^2
 \bigl[p(\alpha\gamma-\beta^2)
       +r(\gamma\varepsilon-\delta^2)\bigr]
 =-\sigma_3,
$
whereas its constant term is
$
 pq^2r(\alpha\gamma-\beta^2)
       (\delta^2-\gamma\varepsilon)
 =-pq^2r\,d_+d_-=\sigma_4.
$
Let~$\rho_1,\ldots,\rho_4$ denote the four eigenvalues of~$P$. Then
\begin{equation*}
 \widetilde f(X)
 =\prod_{j=1}^4(1+cX^2\rho_j^2).
\end{equation*}
Let~$E_k$ be the $k$-th elementary symmetric polynomial in
$\rho_1^2,\ldots,\rho_4^2$. Newton's identities, or a direct expansion give that
\begin{equation}
 \begin{aligned}
 E_1&=\sigma_1^2-2\sigma_2,\ E_2=\sigma_2^2-2\sigma_1\sigma_3+2\sigma_4,\ E_3=\sigma_3^2-2\sigma_2\sigma_4,\ E_4=\sigma_4^2.
 \end{aligned}
 \label{eq:app-en-E-from-sigma}
\end{equation}
Consequently,
\begin{equation}
 \widetilde f(X)
 =1+cE_1X^2+c^2E_2X^4+c^3E_3X^6+c^4E_4X^8.
 \label{eq:app-en-f-E}
\end{equation}
Substituting \eqref{eq:app-en-sigma} into
\eqref{eq:app-en-E-from-sigma} and using
$d_+=\alpha\gamma-\beta^2$ and
$d_-=\gamma\varepsilon-\delta^2$, we find
\begin{equation}\label{eq:app-en-E1}
 \begin{aligned}
 &E_1=\alpha^2p^2+2\gamma^2q^2+\varepsilon^2r^2, \ E_2={}
 2\alpha\gamma d_+p^2q^2
 +\alpha^2\varepsilon^2p^2r^2
 -2\beta^2\delta^2pq^2r+\gamma^4q^4
 +2\gamma\varepsilon d_-q^2r^2,\\
  &E_3=
 \gamma^2p^2q^4d_+^2
 +2\alpha\varepsilon p^2q^2r^2d_+d_-
 +\gamma^2q^4r^2d_-^2,\ E_4=p^2q^4r^2d_+^2d_-^2.
\end{aligned}
 \end{equation}
The only slightly longer simplification in \eqref{eq:app-en-E1} concerns
the coefficient of~$pq^2r$. Before simplification, it is
\begin{align*}
 -2\alpha\varepsilon\gamma^2
 +2\gamma(\alpha d_-+\varepsilon d_+)-2d_+d_-&=-2\alpha\varepsilon\gamma^2
 +2\gamma\bigl[\alpha(\gamma\varepsilon-\delta^2)
 +\varepsilon(\alpha\gamma-\beta^2)\bigr]\\
 &\quad
 -2(\alpha\gamma-\beta^2)(\gamma\varepsilon-\delta^2)
 =-2\beta^2\delta^2.
\end{align*}
Finally, substitute \eqref{eq:app-en-pqr} and
\eqref{eq:app-en-dpm-explicit}, we have
\begin{align*}
 &2c^2\alpha\gamma d_+p^2q^2
 =
 \frac{c^2a_0^2a_1^2(\eta_0-\eta_1)^2}
 {16\eta_0^2\eta_1^2(\eta_0+\eta_1)^2}
 e^{(2\eta_1+2\eta_0)x},\\
 &c^4E_4
 =
 \frac{c^4a_0^4a_1^2a_{-1}^2
 (\eta_0-\eta_1)^4(\eta_0-\eta_{-1})^4}
 {1024\eta_0^4\eta_1^2\eta_{-1}^2
 (\eta_0+\eta_1)^4(\eta_0+\eta_{-1})^4}
 e^{(2\eta_1+4\eta_0+2\eta_{-1})x}.
\end{align*}
Applying the same substitutions to every monomial in
\eqref{eq:app-en-E1} gives, term by term,
$f_{2}$-$f_{8}$ of Lemma \ref{LE1}. Hence
\begin{equation}
 \widetilde f_{2k}=f_{2k},
 \quad 1\leq k\leq4.
 \label{eq:app-en-f-coeff-equality}
\end{equation}

Next, we prove
\begin{equation*}
 \widetilde g_{i,2k+1}=g_{i,2k+1},
 \ 0\leq k\leq3,
 \quad i=1,0,-1.
\end{equation*}
Set
\begin{equation*}
 K:=cP^2,
 \
 F_k:=c^kE_k.
\end{equation*}
Comparing equal powers of~$t$ in
$
 (I_4+tK)\operatorname{adj}(I_4+tK)
 =\det(I_4+tK)I_4
$
gives the exact fourth-order adjugate expansion
\begin{equation*}
 \operatorname{adj}(I_4+tK)=Q_0+tQ_1+t^2Q_2+t^3Q_3,
\end{equation*}
where
\begin{equation*}
 \begin{aligned}
 Q_0&=I_4,\ Q_1=F_1I_4-K,\ Q_2=F_2I_4-F_1K+K^2,\ Q_3=F_3I_4-F_2K+F_1K^2-K^3.
 \end{aligned}
\end{equation*}
Equivalently, $Q_m=F_mI_4-KQ_{m-1}$. It follows from
\eqref{eq:app-en-scaled-G} that
\begin{equation}
 \widetilde G_{2m+1}=RQ_mAR^T,
 \quad 0\leq m\leq3.
 \label{eq:app-en-G-coeff-Q}
\end{equation}
The complete sparse-matrix products in\eqref{eq:app-en-G-coeff-Q} are
listed below.
\begin{equation}\label{eq:app-en-g1-1}
\begin{aligned}
\widetilde g_{1,1}&=p, \
\widetilde g_{1,3}=c\bigl[2\gamma(\gamma-\beta)pq^2-\delta^2q^2r+\varepsilon^2pr^2\bigr], \ \widetilde g_{1,7}=c^3(\beta-\gamma)^2d_-^2pq^4r^2, \\
\widetilde g_{1,5}&=c^2\bigl[-(\alpha-\beta)^2\delta^2p^2q^2r+\gamma^2(\gamma-\beta)^2pq^4+2\varepsilon(\gamma-\beta)d_-pq^2r^2\bigr],\\
\end{aligned}
\end{equation}
Since~$G_{12}=g_0/\sqrt{2}$, we have
\begin{equation*}
\begin{aligned}
 \widetilde g_{0,1}&=\sqrt{2}q,\  \widetilde g_{0,7}=\sqrt{2}c^3(\alpha-\beta)(\varepsilon-\delta)d_+d_-p^2q^3r^2,\\
\widetilde g_{0,3}&=\sqrt{2}cq\bigl[\alpha(\alpha-\beta)p^2-\beta\delta pr+\gamma^2q^2+\varepsilon(\varepsilon-\delta)r^2\bigr],\\
 \widetilde g_{0,5}&=\sqrt{2}c^2\bigl[\gamma d_+(\alpha-\beta)p^2q^3+\alpha\varepsilon(\alpha-\beta)(\varepsilon-\delta)p^2qr^2-\beta\delta(\beta-\gamma)(\delta-\gamma)pq^3r+\gamma d_-(\varepsilon-\delta)q^3r^2\bigr],\\
\end{aligned}
\end{equation*}
and
\begin{equation}\label{eq:app-en-gm-1}
\begin{aligned}
 \widetilde g_{-1,1}&=r,\  \widetilde g_{-1,7}=c^3(\delta-\gamma)^2d_+^2p^2q^4r,\ \widetilde g_{-1,3}=c\bigl[\alpha^2p^2r-\beta^2pq^2+2\gamma(\gamma-\delta)q^2r\bigr],\\
\widetilde g_{-1,5}&=c^2\bigl[2\alpha d_+(\gamma-\delta)p^2q^2r-\beta^2(\varepsilon-\delta)^2pq^2r^2+\gamma^2(\gamma-\delta)^2q^4r\bigr].
\end{aligned}
\end{equation}
To convert these compact expressions into the form used in Lemma \ref{LE1}, it is
enough to use
\begin{equation}
 \begin{aligned}
 \alpha-\beta
 &=\frac{\eta_0-\eta_1}
 {2\eta_1(\eta_0+\eta_1)},
 &
 \beta-\gamma
 &=\frac{\eta_0-\eta_1}
 {2\eta_0(\eta_0+\eta_1)},\\
 \varepsilon-\delta
 &=\frac{\eta_0-\eta_{-1}}
 {2\eta_{-1}(\eta_0+\eta_{-1})},
 &
 \delta-\gamma
 &=\frac{\eta_0-\eta_{-1}}
 {2\eta_0(\eta_0+\eta_{-1})}.
 \end{aligned}
 \label{eq:app-en-difference-dictionary}
\end{equation}
The first two types of terms in \eqref{eq:app-en-g1-1} give
\begin{align*}
 2c\gamma(\gamma-\beta)pq^2
 &=-\frac{ca_0^2a_1(\eta_0-\eta_1)}
 {4\eta_0^2(\eta_0+\eta_1)}
 e^{(\eta_1+2\eta_0)x},\\
 -c\delta^2q^2r
 &=-\frac{ca_0^2a_{-1}}
 {2(\eta_0+\eta_{-1})^2}
 e^{(2\eta_0+\eta_{-1})x},
\end{align*}
which are the corresponding terms in \eqref{B1}.
Because \eqref{eq:app-en-g1-1}--\eqref{eq:app-en-gm-1} list every
monomial in the adjugate expansion, term-by-term application of
\eqref{eq:app-en-pqr}, \eqref{eq:app-en-dpm-explicit}, and
\eqref{eq:app-en-difference-dictionary} gives all coefficients in
Lemma \ref{LE1}. Therefore,
\begin{equation}
 \widetilde g_{i,2k+1}=g_{i,2k+1},
 \quad 0\leq k\leq3,
 \quad i=1,0,-1.
 \label{eq:app-en-g-coeff-equality}
\end{equation}
Finally, we show that
\begin{equation*}
 \widetilde s_{2k}=s_{2k},
 \ 1\leq k\leq3.
\end{equation*}
At points where~$M(X)$ is invertible, the Schur complement applied to
\eqref{eq:app-en-scaled-B-s} yields
\begin{equation*}
 \det\widetilde B(X)
 =\det M(X)\det\bigl(-RM(X)^{-1}(XA)R^T\bigr).
\end{equation*}
Write
\[
 \widetilde U(X)=RM(X)^{-1}(XA)R^T.
\]
Since~$\widetilde U(X)$ is a second-order matrix,
$\det(-\widetilde U)=\det\widetilde U$, and hence
\begin{equation}
 \widetilde s(X)=-2\widetilde f(X)\det\widetilde U(X).
 \label{eq:app-en-s-det-U}
\end{equation}
On the other hand,
\[
 \widetilde G(X)=\widetilde f(X)\widetilde U(X),
 \
 \det\widetilde G(X)=\widetilde f(X)^2\det\widetilde U(X).
\]
Multiplying \eqref{eq:app-en-s-det-U} by~$\widetilde f(X)$ gives the
polynomial identity
\begin{equation}
 \widetilde s(X)\widetilde f(X)
 =-2\det\widetilde G(X)
 =\widetilde g_0(X)^2
 -2\widetilde g_1(X)\widetilde g_{-1}(X).
 \label{eq:app-en-sf-identity}
\end{equation}
The identity is first obtained where~$M(X)$ is invertible. Both sides are
polynomials in the matrix entries, however, so it holds for all parameter
values.

Since~$\widetilde f_0=1$, \eqref{eq:app-en-sf-identity} determines every
$\widetilde s_{2m}$ uniquely and successively. Since
\begin{equation}\label{eq:app-en-s2}
\begin{aligned}
 \widetilde s_2&=\widetilde g_{0,1}^2-2\widetilde g_{1,1}\widetilde g_{-1,1}=2q^2-2pr,\\
 \widetilde s_4
 &={}2\widetilde g_{0,1}\widetilde g_{0,3}
 -2\bigl(
 \widetilde g_{1,1}\widetilde g_{-1,3}
 +\widetilde g_{1,3}\widetilde g_{-1,1}
 \bigr)
 -\widetilde s_2\widetilde f_2
 \\
 &=2cq^2\bigl[
 p^2(\alpha-\beta)^2
 -2pr(\beta-\gamma)(\delta-\gamma)
 +r^2(\varepsilon-\delta)^2
 \bigr],\\
 \widetilde s_6
 &={}2\widetilde g_{0,1}\widetilde g_{0,5}
 +\widetilde g_{0,3}^2
 -2\bigl(
 \widetilde g_{1,1}\widetilde g_{-1,5}
 +\widetilde g_{1,3}\widetilde g_{-1,3}
 +\widetilde g_{1,5}\widetilde g_{-1,1}
 \bigr)
 -\widetilde s_2\widetilde f_4
 -\widetilde s_4\widetilde f_2\\
 &=2c^2q^2\bigl[
 p^2r^2(\alpha-\beta)^2(\varepsilon-\delta)^2
 -pq^2r(\beta-\gamma)^2(\delta-\gamma)^2
 \bigr].
 \end{aligned}
\end{equation}
Returning to the original variables, \eqref{eq:app-en-s2}
become $s_{2}-s_{6}$ in Lemma \ref{LE1}.
Thus
\begin{equation} \label{eq:app-en-s-coeff-equality}
 \widetilde s_{2k}=s_{2k},
 \ 1\leq k\leq3.
\end{equation}
The preceding finite matrix expansions already prove equality of all
coefficients directly. We now explain why the same matching can also be viewed
as order-by-order uniqueness for the cascade equations.

Suppose that two collections of coefficients agree at all lower orders and
both satisfy \eqref{eq9.1d} at the current order. For a denominator
coefficient, their difference~$h=\widetilde f_{2m}-f_{2m}$ satisfies
\begin{equation}
 h''=0.
 \label{eq:app-en-f-difference-ode}
\end{equation}
For every higher-order denominator, the cascade construction imposes the
left-end normalization
\[
 f_{2m}(x),\ f_{2m}'(x)\longrightarrow0
 \ (x\to-\infty).
\]
This excludes both the constant and linear solutions of
\eqref{eq:app-en-f-difference-ode}, so~$h=0$.

For the auxiliary coefficient~$s_{2m}$, \eqref{eq9.1d} is an algebraic recurrence.
At the current order its left-hand side contains
$s_{2m}f_0=s_{2m}$, while every other term contains only lower-order
coefficients. Since~$f_0=1$, equality at all lower orders immediately implies
equality of the current~$s_{2m}$ coefficients.

For any~$i\in\{1,0,-1\}$, the difference between two numerator coefficients,
$h_i=\widetilde g_{i,2m+1}-g_{i,2m+1}$, satisfies
\begin{equation*}
 h_i''-\eta_i^2h_i=0.
\end{equation*}
Consequently,
\[
 h_i=A_ie^{\eta_ix}+B_ie^{-\eta_ix}.
\]
Left-end decay excludes~$B_ie^{-\eta_ix}$. The convention that no new
leading-order homogeneous mode is introduced at higher cascade orders excludes
$A_ie^{\eta_ix}$. Therefore~$h_i=0$. Induction now shows that the cascade
coefficients with the same first-order leading terms and the same construction
normalization are unique at every order. Combining \eqref{eq:app-en-f-coeff-equality},
\eqref{eq:app-en-g-coeff-equality}, and
\eqref{eq:app-en-s-coeff-equality}, we finally obtain
\begin{equation*}
 \widetilde f=f,
 \
 \widetilde g_i=g_i\ (i=1,0,-1),
 \
 \widetilde s=s.
\end{equation*}
In particular, the fourth-order
determinant representation is not a new candidate family: it is an exact
finite encoding of the finite cascade functions constructed in Lemma \ref{LE1}.
\end{proof}
\begin{remark}
The algebraic matching between the determinant coefficients and the explicit
coefficients in Section~9 uses only the definitions of~$A$ and~$C$; it does not
require the spectral relation
$\eta_1^2+\eta_{-1}^2=2\eta_0^2$. That spectral relation is needed to prove
that these finite functions satisfy the bilinear cascade and all of its
full-order residual identities. It is not needed for the purely algebraic
identity between the determinant expansion and the previously obtained
explicit coefficients.
\end{remark}
\begin{lemma}\label{LE6}
For any \(a \in \mathbb{R}^3\), \eqref{T1} is a globally smooth real-valued \(H^1\) solution of the original system, and its left amplitude is exactly the given \eqref{eq9.4}.
\end{lemma}
\begin{proof}
The positivity Lemma \ref{LE2} guarantees no poles on the entire axis. Lemma \ref{LE5} guarantee that the original equation is satisfied; Lemma \ref{LE3} and Lemma \ref{LE4} guarantee that the function and its first derivative are square-integrable at both ends. Smoothness on compact intervals supplements the global \(H^1\) property. The left amplitude is immediately recovered from Lemma \ref{LE3}.
\end{proof}
\section{Proof of Theorem \ref{The1}}\label{S3}
From this section onwards, we turn to the exhaustion proof. The solutions to be studied below are arbitrary real-valued \(H^1\) solutions of the original system, and we do not assume they come from finite Hirota expansions or determinant formulas. In the following, \(u = (u_1, u_0, u_{-1})\) is an arbitrary real-valued $(H^1(\mathbb{R}))^{3}$ solution, and we do not use the conclusion that it has a determinant form at all.
\begin{lemma}\label{LE8}
Every real-valued \(H^1(\mathbb{R})^3\) solution has a unique real triple \(\mathbf{a}\), and
\begin{equation*}
\begin{aligned}
u_i(x) &= a_i e^{\eta_i x} + O(e^{(\eta_i + 2\eta_*)x}), \\
u_i'(x) &= \eta_i a_i e^{\eta_i x} + O(e^{(\eta_i + 2\eta_*)x}).
\end{aligned}
\end{equation*}
\end{lemma}

\begin{proof}
Temporarily denote the cubic term in the original equation as \(F_i(u)\), whose complete expression is given in \begin{equation}\label{M23}
\begin{aligned}
F_1(u) &= -2cu_1^3 - 2cu_0^2u_1 - cu_{-1}u_0^2, \\
F_0(u) &= -cu_0^3 - 2cu_1^2u_0 - 2cu_{-1}^2u_0 - 2cu_1u_{-1}u_0, \\
F_{-1}(u) &= -2cu_{-1}^3 - 2cu_0^2u_{-1} - cu_1u_0^2.
\end{aligned}
\end{equation} The 1D embedding and product estimates give
\[
H^1(\mathbb{R}) \hookrightarrow L^\infty(\mathbb{R}) \cap C_0(\mathbb{R}), \quad \|u_j u_k u_l\|_{L^2} \le \|u_j\|_\infty \|u_k\|_\infty \|u_l\|_{L^2}.
\]
Therefore
\[
u_i'' = \eta_i^2 u_i + F_i(u) \in L^2, \quad u_i \in H^2, \quad u_i' \in H^1 \hookrightarrow C_0.
\]
Repeatedly substituting into the polynomial equation yields
\[
u_i \in C^\infty(\mathbb{R}), \ \lim_{x \to \pm\infty} u_i(x) = \lim_{x \to \pm\infty} u_i'(x) = 0.
\]
The central equation can be written as a homogeneous scalar equation
\begin{equation}\label{M24}
u_0'' - \eta_0^2 u_0 = q_0 u_0, \ q_0 = -cu_0^2 - 2c(u_1^2 + u_{-1}^2 + u_1 u_{-1}).
\end{equation}
Each squared term belongs to \(L^1\), and \(\int |u_1 u_{-1}| \le \|u_1\|_{L^2} \|u_{-1}\|_{L^2}\), and
\[
\|q_0\|_{L^1} \le c\|u_0\|_2^2 + 2c\|u_1\|_2^2 + 2c\|u_{-1}\|_2^2 + 2c\|u_1\|_2 \|u_{-1}\|_2 < \infty.
\]
So \(q_0 \in C(\mathbb{R}) \cap L^1(\mathbb{R})\). By Lemma \ref{LE7}, there exists \(a_0\) such that
\begin{equation}\label{M25}
u_0 = a_0 e^{\eta_0 x} + o(e^{\eta_0 x}), \ u_0' = \eta_0 a_0 e^{\eta_0 x} + o(e^{\eta_0 x}).
\end{equation}
And since the non-homogeneous term of this equation is zero, \(a_0 = 0\) if and only if \(u_0 \equiv 0\).

\noindent Next, write the first and third equations as
\begin{equation}\label{M26}
u_1'' - \eta_1^2 u_1 = q_1 u_1 + h_1, \ q_1 = -2c(u_1^2 + u_0^2), \ h_1 = -cu_1 u_0^2,
\end{equation}
\begin{equation*}
u_{-1}'' - \eta_{-1}^2 u_{-1} = q_{-1} u_{-1} + h_{-1}, \ q_{-1} = -2c(u_{-1}^2 + u_0^2), \ h_{-1} = -cu_1 u_0^2.
\end{equation*}

\noindent The \(q_i\) on both sides are also continuously integrable. Using only the boundedness of the two side components and the just-obtained \(u_0 = O(e^{\eta_0 x})\), we have
\[
|h_1| \le c\|u_1\|_\infty |u_0|^2 = O(e^{2\eta_0 x}), \ |h_{-1}| \le c\|u_1\|_\infty |u_0|^2 = O(e^{2\eta_0 x}).
\]
From \eqref{M13},
\[
2\eta_0 = \eta_1 + \sigma_1 = \eta_{-1} + \sigma_{-1}, \ \sigma_1 := 2\eta_0 - \eta_1 > 0, \quad \sigma_{-1} := 2\eta_0 - \eta_{-1} > 0.
\]
Therefore, Lemma \ref{LE7} can be applied separately. We obtain unique \(a_1, a_{-1}\) satisfying the asymptotic formulas corresponding to \eqref{M25}. The treatment of the two side components only uses the mutual boundedness, without using the not-yet-proven fine decay, so the argument involves no circularity. Nor do we add any ratio assumptions such as \(\eta_{\max} < 3\eta_{\min}\).
Record the complete cubic terms on the right-hand side as \eqref{M23}. So \(u_i'' - \eta_i^2 u_i = F_i(u)\). In the previous subsection, we have given \(u_i = O(e^{\eta_i x})\). The ordinary terms satisfy
\[
3\eta_i - \eta_i = 2\eta_i \geq 2\eta_*, \ (\eta_i + 2\eta_j) - \eta_i = 2\eta_j \geq 2\eta_*.
\]
The weighted indices of the three exchange terms respectively satisfy
\[
\begin{aligned}
(2\eta_0 + \eta_{-1}) - \eta_1 &\geq 2\eta_{-1} \geq 2\eta_*, \\
(2\eta_0 + \eta_1) - \eta_{-1} &\geq 2\eta_1 \geq 2\eta_*, \\
(\eta_1 + \eta_0 + \eta_{-1}) - \eta_0 &= \eta_1 + \eta_{-1} \geq 2\eta_*.
\end{aligned}
\]
Therefore, every monomial term in \eqref{M23} satisfies the uniform estimate for its respective component
\begin{equation}\label{M27}
F_i(u(x)) = O(e^{(\eta_i + 2\eta_*)x}) \quad (x \to -\infty). \tag{6.6}
\end{equation}
Here for \(x \leq 0\), the larger the index rate, the smaller the exponential function. The integrating factor identity is
\begin{align}\label{M29}
\left(e^{\eta_i x}(u_i' - \eta_i u_i)\right)' = e^{\eta_i x}F_i(u(x)).
\end{align}

\noindent Since
\[
\lim_{t \to -\infty} e^{\eta_i t}(u_i'(t) - \eta_i u_i(t)) = 0,
\]

\noindent integrating from \(-\infty\) yields
\[
\begin{aligned}
u_i' - \eta_i u_i = e^{-\eta_i x} \int_{-\infty}^{x} e^{\eta_i t}F_i(u(t)) \, dt= O(e^{(\eta_i + 2\eta_*)x}).
\end{aligned}
\]

\noindent Moreover,
\[
(e^{-\eta_i x}u_i)' = e^{-\eta_i x}(u_i' - \eta_i u_i) = O(e^{2\eta_*x}).
\]

\noindent From \(a_i = \lim_{x \to -\infty} e^{-\eta_i x}u_i(x)\),
\begin{align}\label{M28}
e^{-\eta_i x}u_i(x) - a_i &= \int_{-\infty}^{x} (e^{-\eta_i t}u_i(t))' \, dt = O(e^{2\eta_*x}).
\end{align}
Multiplying back by \(e^{\eta_i x}\), and combining with \eqref{M28}, we obtain Lemma \ref{LE8}.
Indeed, integrating \eqref{M29} again, we have
\[
e^{-\eta_i x}u_i(x) - a_i = \int_{-\infty}^{x} e^{-2\eta_i z} \left( \int_{-\infty}^{z} e^{\eta_i t}F_i(u(t)) \, dt \right) dz.
\]

\noindent Here \(a_i\) is the already proven left limit, not a new free constant added in this step. First, applying Tonelli's theorem to the absolute values, we have
\[
\begin{aligned}
&\int_{-\infty}^{x} e^{-2\eta_i z} \int_{-\infty}^{z} e^{\eta_i t}|F_i(u(t))| \, dt \, dz \\
&= \int_{-\infty}^{x} e^{\eta_i t}|F_i(u(t))| \left( \int_{t}^{x} e^{-2\eta_i z} dz \right) dt \le \frac{1}{2\eta_i} \int_{-\infty}^{x} e^{-\eta_i t}|F_i(u(t))| \, dt < \infty.
\end{aligned}
\]

\noindent The final integrability comes from the left end \(e^{-\eta_i t}F_i(u(t)) = O(e^{2\eta_* t})\), as well as the continuity of the integrand on any finite interval. Therefore, Fubini's theorem can be used to exchange the order of integration. Using
\[
e^{\eta_i x} e^{\eta_i t} \int_{t}^{x} e^{-2\eta_i z} dz = \frac{e^{\eta_i (x-t)} - e^{-\eta_i (x-t)}}{2\eta_i} = \frac{1}{\eta_i} \sinh(\eta_i (x-t)),
\]

\noindent we obtain
\begin{equation}\label{M30}
u_i(x) = a_i e^{\eta_i x} + \frac{1}{\eta_i} \int_{-\infty}^{x} \sinh(\eta_i (x-t)) F_i(u(t)) \, dt.
\end{equation}
It holds for every finite \(x \in \mathbb{R}\), and is an exact equality, not just an approximation valid near negative infinity.
\end{proof}
\begin{proof}[Proof of Theorem \ref{The1}]

Let \(u, v\) be any two real-valued \(H^1\) global solutions, and assume they have the same left amplitude triple. Let \(w_i = u_i - v_i\). By Lemma \ref{LE8}, there exist fixed \(b_0 < 0\) and \(B > 0\) such that
\begin{equation}\label{M31}
|u_i(x)|, |v_i(x)| \le B e^{\eta_i x} \quad (x \le b_0), \ w_i = o(e^{\eta_i x}).
\end{equation}

\noindent Define
\begin{equation}\label{M32}
Z(x) = \max_{j=1,0,-1} e^{-\eta_j x} |w_j(x)|.
\end{equation}

\noindent For \(j, k, l \in \{1, 0, -1\}\), denote
\begin{equation}\label{M33}
\Delta_{jkl} := u_j u_k u_l - v_j v_k v_l = w_j u_k u_l + v_j w_k u_l + v_j v_k w_l.
\end{equation}
From \eqref{M31} and \(|w_j| \le Z(x) e^{\eta_j x}\),
\begin{equation}\label{M34}
|\Delta_{jkl}| \le 3B^2 Z(x) e^{(\eta_j + \eta_k + \eta_l)x}.
\end{equation}

\noindent The three nonlinear differences are exactly written as
\[
\begin{aligned}
F_1(u) - F_1(v) &= -2c\Delta_{111} - 2c\Delta_{001} - c\Delta_{-1,0,0}, \\
F_0(u) - F_0(v) &= -c\Delta_{000} - 2c\Delta_{110} - 2c\Delta_{-1,-1,0} - 2c\Delta_{1,-1,0}, \\
F_{-1}(u) - F_{-1}(v) &= -2c\Delta_{-1,-1,-1} - 2c\Delta_{0,0,-1} - c\Delta_{1,0,0}.
\end{aligned}
\]
For the ordinary terms, after removing the target weight \(e^{\eta_i x}\), the remainder is \(2\eta_j \ge 2\eta_*\). The three types of exchange terms respectively satisfy
\begin{equation}\label{M35}
\begin{aligned}
2\eta_0 + \eta_{-1} - \eta_1 &\ge 2\eta_{-1} \ge 2\eta_*, \\
2\eta_0 + \eta_1 - \eta_{-1} &\ge 2\eta_1 \ge 2\eta_*, \\
\eta_1 + \eta_{-1} &\ge 2\eta_*.
\end{aligned}
\end{equation}

\noindent Adding the terms from \eqref{M34} together, we obtain
\begin{equation}\label{M36}
|F_i(u(x)) - F_i(v(x))| \le C e^{(\eta_i + 2\eta_*)x} Z(x), \quad x \le b_0.
\end{equation}

\noindent The sums of the absolute values of the coefficients for the three equations are \(5c, 7c, 5c\) respectively, so we can uniformly choose
\[
C = 3B^2 \max\{5c, 7c, 5c\} = 21cB^2,
\]

\noindent \(b_0, B, C\) are fixed once here, and \(C\) does not change when the interval endpoint is subsequently moved to the left. The whole process does not divide by any \(a_i\), so zero amplitudes are allowed.

\noindent Subtracting the two solutions in \eqref{M30}, the same-amplitude leading terms cancel out, yielding
\begin{equation}\label{M37}
w_i(x) = \frac{1}{\eta_i} \int_{-\infty}^{x} \sinh(\eta_i(x-t)) [F_i(u(t)) - F_i(v(t))] \, dt.
\end{equation}

\noindent For \(b \le b_0\), let
\[
W_b = \max_i \sup_{x \le b} e^{-\eta_i x} |w_i(x)| < \infty.
\]
The finiteness is obtained from the left asymptotics and the continuity on compact intervals. Using $$\sinh(\eta_i(x-t)) \le \frac{1}{2} e^{\eta_i(x-t)}\ \ \text{for}\ t\leq x,  $$ we have
\begin{align*}
e^{-\eta_i x} |w_i(x)| &\le \frac{C}{2\eta_i} W_b \int_{-\infty}^{x} e^{2\eta_* t} \, dt = \frac{C}{4\eta_i \eta_*} W_b e^{2\eta_* x} \le \frac{C}{4\eta_*^2} W_b e^{2\eta_* b}.
\end{align*}

\noindent Taking the supremum yields \( W_b \le [C/(4\eta_*^2)] e^{2\eta_* b} W_b\). Since \(C\) does not vary with \(b\), we can take a sufficiently negative \(b\) such that the entire coefficient in the brackets is less than 1, hence \(W_b = 0\). Thus \(u = v\) on \((-\infty, b]\).

To extend the conclusion to the entire axis, note that at a finite point \(b\), the function values and first-order derivatives of the two solutions are the same. Let
\[
Y = (u_1, u_0, u_{-1}, u_1', u_0', u_{-1}') \in \mathbb{R}^6.
\]

\noindent The original system becomes a six-dimensional first-order equation \(Y' = \mathcal{F}(Y)\), where \(\mathcal{F}\) is a polynomial mapping. For any compact set \(K \subset \mathbb{R}^6\),
\[
L_K := \sup_{Y \in K} \|D\mathcal{F}(Y)\| < \infty, \ \|\mathcal{F}(Y) - \mathcal{F}(\tilde{Y})\| \le L_K \|Y - \tilde{Y}\|.
\]
Moreover,
\[
Y_u(b) = Y_v(b).
\]

\noindent Local uniqueness of initial values first gives \(Y_u = Y_v\) in a neighborhood of \(b\); since the initial values are the same at equal interval endpoints, iteratively extending along the common interval of existence of the two solutions yields \(Y_u = Y_v\) on \(\mathbb{R}\). Here, the two solutions are known to exist on the entire axis; it is not the case that the existence on the entire axis is deduced from the fact that ``the right-hand side is a polynomial''.

Within the given range, if two sets of real-valued \(H^1(\mathbb{R})^3\) solutions have the same complete left amplitude triple, then they are identical on the entire real axis.
This conclusion applies to all \(H^1\) solutions, not just to the candidate solutions already written in determinant form. It is precisely the realization of the left-uniqueness method from reference \cite{Guo26} in our system; the new work is to verify \eqref{M36} containing the exchange terms, rather than adding an additional common waveform assumption.

Take any real-valued \(H^1\) solution \(u\) of the original system. Lemma \ref{LE8} gives its unique left amplitude \(a\). Using this set of \(a\), construct the candidate solution \(u^a\) from \eqref{T1}; Lemma \ref{LE6} has already proven that it is smooth on the entire axis, belongs to \(H^1\), satisfies the original equation, and has the same left amplitude \(a\). Thus,  \(u = u^a\) on the entire axis, so there are no \(H^1\) solutions outside the explicit family.

Conversely, Lemma \ref{LE6}  has already proven that every set of \(a \in \mathbb{R}^3\) yields a solution. If two sets of parameters yield the same solution, taking the left limit via \eqref{H1} shows that they are identical. Therefore, the parameters and the solution are in one-to-one correspondence. The zero parameters obviously give the zero solution; if the solution is zero, its left amplitudes are also all zero. The screening conditions for the three components will be proven in the next section, thus completing the parametrization classification of the entire solution space.

\end{proof}
\section{Proof of Theorem \ref{The2}}\label{S4}
Set
\begin{equation}
 m:=\frac{M}{N},
 \
 \kappa:=\frac{cN}{2}.
 \label{eq:normalized-data}
\end{equation}
It is easy to get
$-N\leq M\leq N$. Thus  $-N\leq M\leq N$ and $c,N>0$ imply
$
 |m|\le1,
 \
 \kappa>0.
$

By the real spectral theorem, there are an orthogonal matrix $O$ and real
eigenvalues $\alpha_1,\alpha_2$ such that
$
 A=O diag(\alpha_1,\alpha_2)O^T.
$
Then
$A^2=O diag(\alpha_1^2,\alpha_2^2)O^T$, and hence
\begin{align*}
 I_2+\frac{c}{4\eta^2}e^{2\eta x}A^2
 &=O diag\left(
 1+\frac{c\alpha_1^2}{4\eta^2}e^{2\eta x},
 1+\frac{c\alpha_2^2}{4\eta^2}e^{2\eta x}
 \right)O^T.
\end{align*}
Because $O^{-1}=O^T$, its inverse is obtained by taking the reciprocal of each
diagonal entry.  Substitution into~\eqref{N4} gives
\begin{equation}
 U_A(x)=O diag(w_{\alpha_1}(x),w_{\alpha_2}(x))O^T,
 \label{eq:UA-scalar-decomp}
\end{equation}
where
\begin{equation}
 w_\alpha(x):=
 \frac{\alpha e^{\eta x}}
 {1+\dfrac{c\alpha^2}{4\eta^2}e^{2\eta x}}.
 \label{eq:w-alpha-rational}
\end{equation}
If $\alpha=0$, this function is identically zero.  If $\alpha\ne0$, define
$
 x_\alpha:=\frac1\eta
 \log\frac{2\eta}{\sqrt c\,|\alpha|}.
$
Then
$
 z:=\frac{\sqrt c\,|\alpha|}{2\eta}e^{\eta x}
   =e^{\eta(x-x_\alpha)}.
$
Using $2z/(1+z^2)=\operatorname{sech}(\log z)$ in
\eqref{eq:w-alpha-rational}, we find
\begin{equation}\label{eq:w-alpha-sech}
 w_\alpha(x)=sgn(\alpha)\frac{\eta}{\sqrt c}
 \operatorname{sech}\bigl(\eta(x-x_\alpha)\bigr).
\end{equation}
Therefore
\begin{align}
 \int_{\R}w_\alpha(x)^2\,dx
 &=\frac{\eta^2}{c}\int_{\R}
 \operatorname{sech}^2\bigl(\eta(x-x_\alpha)\bigr)\,dx=\frac{\eta}{c}\int_{\R}\operatorname{sech}^2 y\,dy=\frac{\eta}{c}\,[\tanh y]_{-\infty}^{+\infty}
 =\frac{2\eta}{c},
 \quad(\alpha\ne0).
 \label{eq:scalar-mass}
\end{align}
\subsection{The matrix-valued mass identity}

Since~\eqref{eq:UA-scalar-decomp} is symmetric,
\[
 U_A(x)^2
 =Odiag(w_{\alpha_1}(x)^2,w_{\alpha_2}(x)^2)O^T.
\]
The entries are integrable by~\eqref{eq:scalar-mass}, so integration entry by
entry gives
\begin{equation}
 J:=\int_{\R}U_A(x)^2\,dx
 =\frac{2\eta}{c}
 Odiag\bigl(\mathbf{1}_{\{\alpha_1\ne0\}},
              \mathbf{1}_{\{\alpha_2\ne0\}}\bigr)O^T.
 \label{eq:J-lambda-zero}
\end{equation}
Writing out the square of $U$,
\begin{equation*}
 U^2=
 \begin{pmatrix}
 u_1^2+u_0^2/2&(u_1+u_{-1})u_0/\sqrt2\\
 (u_1+u_{-1})u_0/\sqrt2&u_{-1}^2+u_0^2/2
 \end{pmatrix}.
\end{equation*}
Consequently,
\begin{equation}
 N=J_{11}+J_{22}=tr J,
 \
 M=J_{11}-J_{22}.
 \label{eq:NM-from-J}
\end{equation}
\begin{theorem}[Complete constraint classification for $\lambda=0$]
\label{thm:lambda-zero}
Fix $c>0$, $N>0$, and $|M|\le N$, and define $m,\kappa$ by
\eqref{eq:normalized-data}.  Among solutions with $\lambda=0$ and $\mu<0$,
the constraints~\eqref{eq9.4}--\eqref{N1} hold exactly in the following
two cases.

\textbf{Case one.}  For every $|m|\le1$, and
\begin{equation}
 \lambda=0,
 \
 \mu=-\kappa^2=-\frac{c^2N^2}{4}.
 \label{eq:rank-one-multipliers}
\end{equation}
All amplitudes are
\begin{equation}
 a_1=\frac{s}{2}(1+m),
 \quad
 a_0=\frac{\varepsilon s}{\sqrt2}\sqrt{1-m^2},
 \quad
 a_{-1}=\frac{s}{2}(1-m),
 \label{eq:rank-one-amplitudes}
\end{equation}
where $s\in\R\setminus\{0\}$ and $\varepsilon\in\{1,-1\}$.  When $|m|=1$,
$a_0=0$ and the two choices of $\varepsilon$ give the same matrix.

 \textbf{Case two.}  If $M=0$, and
\begin{equation}
 \lambda=0,
 \
 \mu=-\frac{\kappa^2}{4}=-\frac{c^2N^2}{16}.
 \label{eq:rank-two-multipliers}
\end{equation}
The amplitude matrix may be any invertible real symmetric matrix, equivalently
\begin{equation}
 a_1a_{-1}-\frac{a_0^2}{2}\ne0.
 \label{eq:rank-two-condition}
\end{equation}
\end{theorem}
\begin{proof}
Because $N>0$, the zero matrix $A=0$ is impossible: it would give $U_A=0$ and
$N=0$.  A nonzero $2\times2$ matrix has rank one or two, so the two cases below
are exhaustive.

A real symmetric rank-one matrix has exactly one nonzero eigenvalue.  If that
eigenvalue is denoted by $s\ne0$ and $v=(v_1,v_2)^T$ is a corresponding unit
eigenvector, then
\begin{equation*}
 A=svv^T,
 \
 v_1^2+v_2^2=1.
\end{equation*}
Formula~\eqref{eq:J-lambda-zero} becomes
$
 J=\frac{2\eta}{c}vv^T.
$
Thus, by~\eqref{eq:NM-from-J}, we have
\begin{align*}
 N&=tr J
   =\frac{2\eta}{c}tr(vv^T)
   =\frac{2\eta}{c}(v_1^2+v_2^2)
   =\frac{2\eta}{c},\\
 M&=J_{11}-J_{22}
   =\frac{2\eta}{c}(v_1^2-v_2^2).
\end{align*}
Since $\kappa=cN/2$, the first identity gives $\eta=\kappa$.  Since
$\mu=-\eta^2$, this proves~\eqref{eq:rank-one-multipliers}.  Dividing the
second identity by the first gives
\begin{equation}
 v_1^2-v_2^2=\frac{M}{N}=m.
 \label{eq:v-difference}
\end{equation}
Solving the linear system formed by~\eqref{eq:v-difference} and
$v_1^2+v_2^2=1$ gives
\begin{equation}
 v_1^2=\frac{1+m}{2},
 \
 v_2^2=\frac{1-m}{2}.
 \label{eq:v-squares}
\end{equation}
Real $v_1,v_2$ exist exactly when $|m|\le1$.  Furthermore,
\[
 |2v_1v_2|
 =2\sqrt{v_1^2v_2^2}
 =2\sqrt{\frac{1-m^2}{4}}
 =\sqrt{1-m^2}.
\]
Hence, away from the endpoints, there is a sign $\varepsilon\in\{1,-1\}$ such that
\begin{equation}
 2v_1v_2=\varepsilon\sqrt{1-m^2}.
 \label{eq:v-cross}
\end{equation}
At $|m|=1$ the left and right sides are both zero, independently of $\varepsilon$.
Finally, comparing
\[
 svv^T
 =\begin{pmatrix}
 sv_1^2&sv_1v_2\\sv_1v_2&sv_2^2
 \end{pmatrix}
 =\begin{pmatrix}
 a_1&a_0/\sqrt2\\a_0/\sqrt2&a_{-1}
 \end{pmatrix}
\]
and substituting~\eqref{eq:v-squares}--\eqref{eq:v-cross} yields exactly
\eqref{eq:rank-one-amplitudes}.

Conversely, let $s\ne0$, $|m|\le1$, and choose amplitudes from
\eqref{eq:rank-one-amplitudes}.  A direct determinant calculation gives
\begin{align*}
 \det A=a_1a_{-1}-\frac{a_0^2}{2}=\frac{s^2}{4}(1-m^2)
   -\frac12\left(\frac{s^2}{2}(1-m^2)\right)=0.
\end{align*}
The matrix is not zero because $s\ne0$ and $(1+m)+(1-m)=2$.  Thus it has rank
one.  Standing \eqref{eq9.4}  produces a global $H^1$
solution, and the preceding mass-matrix calculation gives the prescribed
$N,M$.

Both eigenvalues are nonzero.  The two indicators in
\eqref{eq:J-lambda-zero} are therefore one, and orthogonality gives
\begin{equation}
 J=\frac{2\eta}{c}OI_2O^T=\frac{2\eta}{c}I_2.
 \label{eq:J-rank-two}
\end{equation}
Hence
$
 M=J_{11}-J_{22}=0,
 \
 N=tr J=\frac{4\eta}{c}.
$
Thus rank two is possible only for $M=0$, and
$\eta=cN/4=\kappa/2$.  Since $\mu=-\eta^2$, we obtain
\eqref{eq:rank-two-multipliers}.  A $2\times2$ matrix has rank two exactly when
its determinant is nonzero, which is precisely~\eqref{eq:rank-two-condition}.

Conversely, every invertible symmetric $A$ has rank two.  With
$\eta=\kappa/2$, formula~\eqref{eq:J-rank-two} gives
$N=4\eta/c$ and $M=0$.  Standing again gives
the required global $H^1$ solution.  This proves necessity, sufficiency, and
exhaustion.
\end{proof}
Let
\begin{equation*}
 K:=diag(\eta_1,\eta_0,\eta_0,\eta_{-1}),
 \
 R:=\begin{pmatrix}1&1&0&0\\0&0&1&1\end{pmatrix},
\end{equation*}
From \eqref{T1}, we have
\begin{equation}
 U_a=R\bigl(I_4+c\mathcal AC\mathcal AC\bigr)^{-1}
 \mathcal A R^T.
 \label{eq:det-solution}
\end{equation}
Define the three indicators
\begin{equation*}
 \varepsilon_i:=\mathbf{1}_{\{a_i\ne0\}},
 \ i=1,0,-1.
\end{equation*}
The active coordinate space $E\subset\R^4$ is defined by retaining coordinate
$1$ when $a_1\ne0$, both coordinates $2,3$ when $a_0\ne0$, and coordinate $4$
when $a_{-1}\ne0$.  Thus
$
 \dim E=\varepsilon_1+2\varepsilon_0+\varepsilon_{-1}.
$
After ordering active coordinates first,
$
 \mathcal A=\begin{pmatrix}\mathcal A_E&0\\0&0\end{pmatrix}.
$
Every retained diagonal entry is nonzero, and the retained middle
$2\times2$ block is
$
 \begin{pmatrix}0&q\\q&0\end{pmatrix},
 \ q=\frac{a_0}{\sqrt2}e^{\eta_0x}\ne0,
$
whose determinant is $-q^2\ne0$.  Hence $\mathcal A_E$ is invertible.  Denote
by $C_E,K_E$ the corresponding principal submatrices and by $R_E$ the retained
columns of $R$.

The inactive rows of $\mathcal A$ are zero.  Multiplying the block matrices in
\eqref{eq:det-solution}, or equivalently taking the Schur complement of the
inactive identity block, gives
\begin{equation}
 U_\mathbf{a}=R_E\bigl(I_E+c\mathcal A_EC_E\mathcal A_EC_E\bigr)^{-1}
 \mathcal A_E R_E^T.
 \label{eq:active-compression}
\end{equation}
No inverse of the possibly singular full matrix $\mathcal A$ has been used. Because the two middle diagonal entries of $K$ are equal,
\begin{equation}
 \mathcal A_E'=K_E\mathcal A_E=\mathcal A_EK_E.
 \label{eq:A-derivative}
\end{equation}
Next, each nonzero entry of $C$ has the form $(k_i+k_j)^{-1}$, where $k_i$
and $k_j$ are the corresponding diagonal entries of $K$.  Thus multiplication
by $K$ on the left and on the right yields
$$
 KC+CK=R^TR.
$$
Indeed, the two diagonal $2\times2$ blocks on the left consist entirely of
ones, and all off-block entries are zero, exactly as in $R^TR$.  Restriction to
$E$ gives
\begin{equation}
 K_EC_E+C_EK_E=R_E^TR_E.
 \label{eq:Lyapunov-E}
\end{equation}
For notational brevity, set
\begin{align}\label{eq:T-def}
 T_E:=\bigl(\mathcal A_E^{-1}
              +cC_E\mathcal A_EC_E\bigr)^{-1},\ P_E:=\bigl(\mathcal A_E^{-1}C_E^{-1}\mathcal A_E^{-1}+cC_E\bigr)^{-1}.
\end{align}
Since
\[
 I_E+c\mathcal A_EC_E\mathcal A_EC_E
 =\mathcal A_E
  (\mathcal A_E^{-1}+cC_E\mathcal A_EC_E),
\]
equation~\eqref{eq:active-compression} becomes
\begin{equation}
 U_a=R_ET_ER_E^T.
 \label{eq:U-via-T}
\end{equation}

\begin{lemma}
\label{lem:derivative-identity}
For every $x\in\R$,
\begin{equation}
 U_a(x)^2=\frac{d}{dx}\bigl(R_EP_E(x)R_E^T\bigr).
 \label{eq:exact-derivative}
\end{equation}
\end{lemma}

\begin{proof}
Put
\[
 B_E:=P_E^{-1}
 =\mathcal A_E^{-1}C_E^{-1}\mathcal A_E^{-1}+cC_E.
\]
From~\eqref{eq:A-derivative} and the inverse differentiation rule,
\begin{equation*}
 (\mathcal A_E^{-1})'
 =-\mathcal A_E^{-1}\mathcal A_E'\mathcal A_E^{-1}
 =-\mathcal A_E^{-1}K_E
 =-K_E\mathcal A_E^{-1}.
 \label{eq:Ainv-derivative}
\end{equation*}
The final equality again uses the commutation of $K_E$ and $\mathcal A_E$.
Since $C_E$ is constant in $x$,
\begin{align*}
 B_E'
 &=(\mathcal A_E^{-1})'C_E^{-1}\mathcal A_E^{-1}
   +\mathcal A_E^{-1}C_E^{-1}(\mathcal A_E^{-1})'\\
 &=-\mathcal A_E^{-1}
   (K_EC_E^{-1}+C_E^{-1}K_E)\mathcal A_E^{-1}.
\end{align*}
Multiplying~\eqref{eq:Lyapunov-E} on the left and right by $C_E^{-1}$ gives
\begin{equation*}
 C_E^{-1}K_E+K_EC_E^{-1}
 =C_E^{-1}R_E^TR_EC_E^{-1}.
\end{equation*}
Therefore, with $D_E:=\mathcal A_E^{-1}C_E^{-1}$,
\begin{equation*}
 B_E'=-D_ER_E^TR_ED_E^T.
 \label{eq:B-derivative}
\end{equation*}
Differentiating $P_E=B_E^{-1}$ now gives
\begin{equation}
 P_E'=-P_EB_E'P_E
      =P_ED_ER_E^TR_ED_E^TP_E.
 \label{eq:Pprime-intermediate}
\end{equation}
It remains to identify the two outer factors.  From
\begin{align*}
 D_ET_E^{-1}
 &=\mathcal A_E^{-1}C_E^{-1}
   (\mathcal A_E^{-1}+cC_E\mathcal A_EC_E)\\
 &=\mathcal A_E^{-1}C_E^{-1}\mathcal A_E^{-1}+cC_E
 =B_E,
\end{align*}
we get $P_ED_E=T_E$.  Similarly,
\begin{align*}
 T_E^{-1}D_E^T
 &=(\mathcal A_E^{-1}+cC_E\mathcal A_EC_E)
   C_E^{-1}\mathcal A_E^{-1}\\
 &=\mathcal A_E^{-1}C_E^{-1}\mathcal A_E^{-1}+cC_E
 =B_E,
\end{align*}
so $D_E^TP_E=T_E$.  Substituting both identities into
\eqref{eq:Pprime-intermediate} yields
$
 P_E'=T_ER_E^TR_ET_E.
$
Finally, by~\eqref{eq:U-via-T},
\[
 U_a^2
 =R_ET_ER_E^TR_ET_ER_E^T
 =R_EP_E'R_E^T
 =\frac{d}{dx}(R_EP_ER_E^T),
\]
which proves~\eqref{eq:exact-derivative}.
\end{proof}
\begin{lemma}
\label{lem:P-endpoints}
On every nonempty active space $E$,
\begin{equation}
 \lim_{x\to-\infty}P_E(x)=0,
 \
 \lim_{x\to+\infty}P_E(x)=\frac1cC_E^{-1}.
 \label{eq:P-endpoints}
\end{equation}
\end{lemma}

\begin{proof}
Factoring $P_E^{-1}$ on both sides gives
\begin{align}
 P_E^{-1}
 &=\mathcal A_E^{-1}
   (C_E^{-1}+c\mathcal A_EC_E\mathcal A_E)
   \mathcal A_E^{-1},\notag\\
 P_E
 &=\mathcal A_E
   (C_E^{-1}+c\mathcal A_EC_E\mathcal A_E)^{-1}
   \mathcal A_E.
 \label{eq:P-minus-factor}
\end{align}
As $x\to-\infty$, every entry of $\mathcal A_E(x)$ tends to zero and the
middle inverse in~\eqref{eq:P-minus-factor} tends to $C_E$.  Hence the product
tends to zero.

As $x\to+\infty$, every nonzero entry of $\mathcal A_E$ grows exponentially,
whereas every entry of $\mathcal A_E^{-1}$ decays exponentially.  It follows
directly from~\eqref{eq:T-def} that
\[
 P_E^{-1}
 =\mathcal A_E^{-1}C_E^{-1}\mathcal A_E^{-1}+cC_E
 \longrightarrow cC_E.
\]
Since $C_E$ is positive definite and hence invertible, inversion is continuous
near $cC_E$, giving $P_E\to c^{-1}C_E^{-1}$.
\end{proof}
Integrating Lemma~\ref{lem:derivative-identity} over $[-L,L]$ and applying the
fundamental theorem of calculus gives
\[
 \int_{-L}^{L}U_a(x)^2\,dx
 =R_E\bigl(P_E(L)-P_E(-L)\bigr)R_E^T.
\]
Letting $L\to\infty$ and using Lemma~\ref{lem:P-endpoints}, we obtain
\begin{equation}
 J:=\int_{\R}U_a(x)^2\,dx
 =\frac1cR_EC_E^{-1}R_E^T.
 \label{eq:J-nonzero-lambda}
\end{equation}
\begin{proposition}
\label{prop:mass-formula}
For every solution in the noncollision determinant family,
\begin{align} \label{eq:N-indicator}
 N=\frac2c
 (\eta_1\varepsilon_1+2\eta_0\varepsilon+\eta_{-1}\varepsilon_{-1}),\ M=\frac2c
 (\eta_1\varepsilon_1-\eta_{-1}\varepsilon_{-1}).
\end{align}
These values depend on the support indicators but not on the magnitudes or
signs of the nonzero amplitudes.
\end{proposition}
\begin{proof}
For positive $p\ne q$, define
\begin{equation*}
 C(p,q):=
 \begin{pmatrix}
 (2p)^{-1}&(p+q)^{-1}\\
 (p+q)^{-1}&(2q)^{-1}
 \end{pmatrix}.
 \label{eq:Cpq}
\end{equation*}
Its determinant is
\begin{align*}
 \det C(p,q)
 &=\frac1{4pq}-\frac1{(p+q)^2}=\frac{(p+q)^2-4pq}{4pq(p+q)^2}
 =\frac{(p-q)^2}{4pq(p+q)^2}>0.
 \label{eq:Cpq-det}
\end{align*}
Thus
\[
 C(p,q)^{-1}
 =\frac1{\det C(p,q)}
 \begin{pmatrix}
 (2q)^{-1}&-(p+q)^{-1}\\
 -(p+q)^{-1}&(2p)^{-1}
 \end{pmatrix}.
\]
Consequently,
\begin{align*}
 (1,1)C(p,q)^{-1}(1,1)^T
 &=\frac{\dfrac1{2p}+\dfrac1{2q}-\dfrac2{p+q}}
 {\dfrac{(p-q)^2}{4pq(p+q)^2}}=\frac{\dfrac{(p+q)^2-4pq}{2pq(p+q)}}
 {\dfrac{(p-q)^2}{4pq(p+q)^2}}=\frac{\dfrac{(p-q)^2}{2pq(p+q)}}
 {\dfrac{(p-q)^2}{4pq(p+q)^2}}
 =2(p+q).
\end{align*}
If only the first coordinate is active, the restricted matrix is
$[(2p)^{-1}]$ and its inverse contributes $2p$; if only the second coordinate
is active, the contribution is $2q$; if neither is active, the contribution is
zero.

In~\eqref{M1}, the first row of $R$ selects the block
$C(\eta_1,\eta_0)$ and the second row selects
$C(\eta_0,\eta_{-1})$.  Applying the preceding three possibilities to each
block in~\eqref{eq:J-nonzero-lambda} gives
\begin{equation*}
 J_{11}=\frac2c(\eta_1\varepsilon_1+\eta_0\varepsilon_0),
 \
 J_{22}=\frac2c(\eta_0\varepsilon_0+\eta_{-1}\varepsilon_{-1}).
 \label{eq:J-diagonal}
\end{equation*}
Together with~\eqref{eq:NM-from-J}, this complete the  proof of the Proposition \ref{prop:mass-formula}.
\end{proof}
Multiply~\eqref{eq:N-indicator} by $c/2$ and use
\eqref{eq:normalized-data}.  The constraint problem is equivalent to
\begin{align}\label{eq:algebraic-mass}
 \eta_1\varepsilon_1+2\eta_0\varepsilon_0+\eta_{-1}\varepsilon_{-1}=\kappa,\ \ \eta_1\varepsilon_1-\eta_{-1}\varepsilon_{-1}=\kappa m,\ \ \eta_1^2+\eta_{-1}^2&=2\eta_0^2.
\end{align}
The rates must be strictly positive, and the noncollision condition is
$
 \eta_1\ne\eta_{-1}.
$
There are $2^3-1=7$ nonempty binary triples
$(\varepsilon_1,\varepsilon_0,\varepsilon_{-1})$.  The empty triple would give $N=0$, contrary to
$N>0$.  We now solve all seven cases without omitting any algebra.

Let $(\varepsilon_1,\varepsilon_0,\varepsilon_{-1})=(1,0,0)$.  Equations
\eqref{eq:algebraic-mass}read
$
 \eta_1=\kappa,
\
 \eta_1=\kappa m.
$
Because $\kappa>0$, comparison gives $m=1$.  The unobserved outer rate
$\eta_{-1}$ remains free; write
$
 \eta_{-1}=t>0.
$
Then~\eqref{eq:algebraic-mass} determines
$
 \eta_0=\sqrt{\frac{\kappa^2+t^2}{2}}.
$
Using
\begin{equation}
 \mu=-\frac{\kappa^2+t^2}{2},
 \
 \lambda=\frac{t^2-\kappa^2}{2}.
 \label{eq:A1-multipliers}
\end{equation}
The condition $\lambda\ne0$ is exactly $t\ne\kappa$.  All three rates are
positive, so holds by the converse argument after
.  Hence the complete stratum is
\begin{equation*}
 m=1,\ t>0,\quad t\ne\kappa,\
 (a_1,a_0,a_{-1})=(s_1,0,0),\quad s_1\ne0,
\end{equation*}
with multipliers~\eqref{eq:A1-multipliers}.

Let $(\varepsilon_1,\varepsilon_0,\varepsilon_{-1})=(0,0,1)$.  Then
$
 \eta_{-1}=\kappa,
 \
 -\eta_{-1}=\kappa m,
$
so $m=-1$.  Put $\eta_1=t>0$.  The spectral relation gives
$
 \eta_0=\sqrt{\frac{t^2+\kappa^2}{2}},
$
and hence
$
 \mu=-\frac{\kappa^2+t^2}{2},
 \
 \lambda=\frac{\kappa^2-t^2}{2}.
$
Again $\lambda\ne0$ exactly when $t\ne\kappa$.  Thus
\begin{equation*}
 m=-1,\ t>0,\quad t\ne\kappa,\
 (a_1,a_0,a_{-1})=(0,0,s_{-1}),\ s_{-1}\ne0.
\end{equation*}
Let $(\varepsilon_1,\varepsilon_0,\varepsilon_{-1})=(0,1,0)$.  Then
$
 2\eta_0=\kappa,
 \
 0=\kappa m.
$
Therefore
\begin{equation*}
 \eta_0=\frac\kappa2,
 \
 m=0,
\
 \mu=-\eta_0^2=-\frac{\kappa^2}{4}
 =-\frac{c^2N^2}{16}.
\end{equation*}
From
$
 \eta_1^2=\frac{\kappa^2}{4}-\lambda,
 \
 \eta_{-1}^2=\frac{\kappa^2}{4}+\lambda.
$
Both rates are strictly positive exactly when
$
 -\frac{\kappa^2}{4}<\lambda<\frac{\kappa^2}{4}.
$
Removing the collision value $\lambda=0$ gives
$
 m=0,\
 0<|\lambda|<\frac{\kappa^2}{4},
 \
 (a_1,a_0,a_{-1})=(0,s_0,0),\ s_0\ne0.
$

Let $(\varepsilon_1,\epsilon_0,\varepsilon_{-1})=(1,0,1)$.  The mass and magnetization equations
are
\begin{equation*}
 \eta_1+\eta_{-1}=\kappa,
 \
 \eta_1-\eta_{-1}=\kappa m.
\end{equation*}
Adding and subtracting them gives
\begin{equation*}
 \eta_1=\frac\kappa2(1+m),
 \
 \eta_{-1}=\frac\kappa2(1-m).
\end{equation*}
Both are positive exactly when $-1<m<1$.  The spectral relation gives
\begin{align*}
 \eta_0^2
 &=\frac{\eta_1^2+\eta_{-1}^2}{2}=\frac12\frac{\kappa^2}{4}
   \bigl[(1+m)^2+(1-m)^2\bigr]=\frac{\kappa^2}{8}(2+2m^2)
 =\frac{\kappa^2}{4}(1+m^2),
\end{align*}
so
\begin{equation*}
 \eta_0=\frac\kappa2\sqrt{1+m^2}.
\end{equation*}
The multipliers are
\begin{align*}
 \mu
 &=-\eta_0^2
 =-\frac{\kappa^2}{4}(1+m^2)
 =-\frac{c^2}{16}(N^2+M^2),
 \label{eq:B-mu}\\
 \lambda
 &=\frac{\eta_{-1}^2-\eta_1^2}{2}
 =\frac{(\eta_{-1}-\eta_1)(\eta_{-1}+\eta_1)}2=\frac{(-\kappa m)\kappa}{2}
 =-\frac{\kappa^2m}{2}
 =-\frac{c^2NM}{8}.
\end{align*}
Thus $\lambda\ne0$ exactly when $m\ne0$.  The complete condition and
amplitudes are therefore
\begin{equation*}
 0<|m|<1,
 \
 (a_1,a_0,a_{-1})=(s_1,0,s_{-1}),
 \ s_1s_{-1}\ne0.
\end{equation*}
Let $(\varepsilon_1,\varepsilon_0,\varepsilon_{-1})=(1,1,0)$.  Equations
\eqref{eq:algebraic-mass} become
$
 \eta_1+2\eta_0=\kappa,
 \
 \eta_1=\kappa m.
$
Therefore
\begin{equation}
 \eta_1=\kappa m,
\
 \eta_0=\frac\kappa2(1-m).
 \label{eq:Cplus-first-rates}
\end{equation}
Since $\eta_1>0$, we need $m>0$.  The relation
\eqref{eq:algebraic-mass} gives
\begin{align*}
 \eta_{-1}^2=2\eta_0^2-\eta_1^2=2\left[\frac{\kappa^2}{4}(1-m)^2\right]-\kappa^2m^2=\frac{\kappa^2}{2}(1-2m+m^2-2m^2)=\frac{\kappa^2}{2}(1-2m-m^2).
\end{align*}
Hence
\begin{equation}
 \eta_{-1}=\frac\kappa{\sqrt2}\sqrt{1-2m-m^2}.
 \label{eq:Cplus-last-rate}
\end{equation}
Strict positivity of this rate is equivalent to
\[
 1-2m-m^2>0
 \ \Longleftrightarrow\
 m^2+2m-1<0.
\]
The roots of $m^2+2m-1$ are $-1-\sqrt2$ and $-1+\sqrt2$.
Since the quadratic opens upward, it is negative strictly between its roots.
Intersecting with $m>0$ yields
\begin{equation}
 0<m<\sqrt2-1.
 \label{eq:Cplus-basic-range}
\end{equation}
This range also makes $\eta_0>0$ because $\sqrt2-1<1$.

Now
\begin{align*}
 \mu
 &=-\eta_0^2
 =-\frac{\kappa^2}{4}(1-m)^2
 =-\frac{c^2}{16}(N-M)^2,\\\nonumber
 \lambda
 &=\eta_0^2-\eta_1^2
 =\frac{\kappa^2}{4}(1-m)^2-\kappa^2m^2=\frac{\kappa^2}{4}(1-2m-3m^2)\\\nonumber
& =\frac{\kappa^2}{4}(1+m)(1-3m)=\frac{c^2}{16}(N+M)(N-3M).
\end{align*}
Thus $\lambda=0$ exactly when $m=1/3$ within
\eqref{eq:Cplus-basic-range}.  At that point all three rates equal
$\kappa/3$, so it is precisely the collision case.  Removing it, the complete
stratum is
\begin{equation*}
 0<m<\sqrt2-1,
 \ m\ne\frac13,
 \
 (a_1,a_0,a_{-1})=(s_1,s_0,0),
\ s_1s_0\ne0.
\end{equation*}
Let $(\varepsilon_1,\varepsilon_0,\varepsilon_{-1})=(0,1,1)$.  Now
$
 2\eta_0+\eta_{-1}=\kappa,
 \
 -\eta_{-1}=\kappa m.
$
Hence
\begin{equation}
 \eta_{-1}=-\kappa m,
 \
 \eta_0=\frac\kappa2(1+m).
 \label{eq:Cminus-first-rates}
\end{equation}
The first rate is positive only if $m<0$.  Furthermore,
\begin{align*}
 \eta_1^2=2\eta_0^2-\eta_{-1}^2=\frac{\kappa^2}{2}(1+m)^2-\kappa^2m^2=\frac{\kappa^2}{2}(1+2m-m^2),
\end{align*}
and therefore
\begin{equation}
 \eta_1=\frac\kappa{\sqrt2}\sqrt{1+2m-m^2}.
 \label{eq:Cminus-last-rate}
\end{equation}
The inequality $1+2m-m^2>0$ is equivalent to
$m^2-2m-1<0$.  Its roots are $1-\sqrt2$ and $1+\sqrt2$, so intersection with
$m<0$ gives
\begin{equation}
 1-\sqrt2<m<0.
 \label{eq:Cminus-basic-range}
\end{equation}
In this interval $\eta_0>0$ as well.  Next,
\begin{align*}
 \mu
 &=-\eta_0^2
 =-\frac{\kappa^2}{4}(1+m)^2
 =-\frac{c^2}{16}(N+M)^2,\\\nonumber
 \lambda
 &=\eta_{-1}^2-\eta_0^2
 =\kappa^2m^2-\frac{\kappa^2}{4}(1+m)^2=\frac{\kappa^2}{4}(3m^2-2m-1)
 =-\frac{\kappa^2}{4}(1-m)(1+3m)\notag\\\nonumber
 &=-\frac{c^2}{16}(N-M)(N+3M).
\end{align*}
Within~\eqref{eq:Cminus-basic-range}, $\lambda=0$ exactly when $m=-1/3$.
The complete noncollision stratum is therefore
\begin{equation*}
 1-\sqrt2<m<0,
 \ m\ne-\frac13,
 \
 (a_1,a_0,a_{-1})=(0,s_0,s_{-1}),
 \quad s_0s_{-1}\ne0.
\end{equation*}
Assume $(\varepsilon_1,\varepsilon_0,\varepsilon_{-1})=(1,1,1)$.  The first two algebraic
equations are
\begin{equation*}
 \eta_1+2\eta_0+\eta_{-1}=\kappa,
\
 \eta_1-\eta_{-1}=\kappa m.
\end{equation*}
From the second equation,
$
 \eta_{-1}=\eta_1-\kappa m.
$
Substitution into the first gives
$
 2\eta_1+2\eta_0-\kappa m=\kappa,
$
so
\begin{equation}
 \eta_1=\frac\kappa2(1+m)-\eta_0,
 \
 \eta_{-1}=\frac\kappa2(1-m)-\eta_0.
 \label{eq:D-outers-in-middle}
\end{equation}
Let
$
 A_+:=\frac\kappa2(1+m),
 \
 A_-:=\frac\kappa2(1-m).
$
Then~\eqref{eq:algebraic-mass} becomes
\begin{align*}
 (A_+-\eta_0)^2+(A_--\eta_0)^2&=2\eta_0^2,\ A_+^2+A_-^2-2(A_++A_-)\eta_0+2\eta_0^2=2\eta_0^2.
\end{align*}
Because $A_++A_-=\kappa>0$, cancellation and division by $2\kappa$ give the
unique value
\begin{align}
 \eta_0=\frac{A_+^2+A_-^2}{2\kappa}=\frac{\kappa^2\bigl[(1+m)^2+(1-m)^2\bigr]/4}{2\kappa}
 =\frac\kappa4(1+m^2).
 \label{eq:D-middle}
\end{align}
Substitution into~\eqref{eq:D-outers-in-middle} yields
\begin{align}\label{eq:D-rate-one}
 \eta_1=\frac\kappa4(1+2m-m^2),\
 \eta_{-1}=\frac\kappa4(1-2m-m^2).
\end{align}
The middle rate in~\eqref{eq:D-middle} is always positive.  For the first
outer rate,
\[
 1+2m-m^2>0
 \ \Longleftrightarrow\
 1-\sqrt2<m<1+\sqrt2.
\]
For the other outer rate,
$
 1-2m-m^2>0
 \ \Longleftrightarrow\
 -1-\sqrt2<m<\sqrt2-1.
$
Their intersection is
\begin{equation*}
 1-\sqrt2<m<\sqrt2-1,
 \ \text{equivalently}\
 |m|<\sqrt2-1.
\end{equation*}
Also,
$
 \eta_1-\eta_{-1}=\kappa m,
$
so the noncollision condition is $m\ne0$.  Therefore the exact existence range
is
$
 0<|m|<\sqrt2-1.
$
Finally, $\mu
 =-\eta_0^2
 =-\frac{\kappa^2}{16}(1+m^2)^2
 =-\frac{c^2(N^2+M^2)^2}{64N^2},$
\begin{align*}
 \lambda
 &=\frac{(\eta_{-1}-\eta_1)(\eta_{-1}+\eta_1)}2=\frac{(-\kappa m)\,[\kappa(1-m^2)/2]}2
 =-\frac{\kappa^2m(1-m^2)}4
 =-\frac{c^2M(N^2-M^2)}{16N}.
\end{align*}
All amplitudes on this stratum are arbitrary nonzero real numbers:
\begin{equation*}
 (a_1,a_0,a_{-1})=(s_1,s_0,s_{-1}),
 \ s_1s_0s_{-1}\ne0.
\end{equation*}
\begin{proof}[Proof of Theorem \ref{The2}]
For family~\textup{(i)}, Theorem~\ref{thm:lambda-zero} gives
\[
 U_A=w_s\,vv^T,\
 w_s(x)=\sigma\frac{\kappa}{\sqrt c}
 \operatorname{sech}\bigl(\kappa(x-x_0)\bigr),
\]
with
\[
 v_1^2=\frac{1+m}{2},\
 v_2^2=\frac{1-m}{2},\
 2v_1v_2=\varepsilon\sqrt{1-m^2}.
\]
Since \(m=M/N\) and \(\kappa=cN/2\),
\[
 \frac{\kappa}{\sqrt c}\frac{1+m}{2}
 =\frac{\sqrt c}{4}(N+M),\
 \frac{\kappa}{\sqrt c}\frac{\varepsilon}{\sqrt2}
 \sqrt{1-m^2}
 =\frac{\varepsilon\sqrt c}{2\sqrt2}\sqrt{N^2-M^2},
\]
and the analogous calculation gives
\(\sqrt c(N-M)/4\) for the last component.  These are exactly the three
lines of Case (i).

For family~\textup{(ii)}, diagonalize the invertible symmetric amplitude
matrix as
\[
 A=O\operatorname{diag}(\alpha_1,\alpha_2)O^T.
\]
Every nonzero \(\alpha_j\) can be written uniquely in terms of its sign
\(\sigma_j\) and center \(x_j\) through
\[
 x_j=\frac1\eta\log\frac{2\eta}{\sqrt c\,|\alpha_j|}.
\]
Equations~\eqref{eq:w-alpha-sech} and
\eqref{eq:UA-scalar-decomp} then give
$
 U_A=O\operatorname{diag}(w_1,w_2)O^T.
$
Writing \(O\) as a rotation, with a possible reflection absorbed into the
signs and an exchange of the two labels completes the proof of Case (ii).

In families~\textup{(iii)}--\textup{(vi)}, the active physical equations
decouple into scalar equations.  An outer component with decay rate \(\eta\)
satisfies
\[
 y''-\eta^2y+2cy^3=0,
\]
whose nonzero \(H^1\) solutions are
\[
 y(x)=\sigma\frac{\eta}{\sqrt c}
 \operatorname{sech}\bigl(\eta(x-x_*)\bigr).
\]
An isolated middle component satisfies
\[
 y''-\eta^2y+cy^3=0,
\]
whose nonzero \(H^1\) solutions are
\[
 y(x)=\sigma\frac{\sqrt2\,\eta}{\sqrt c}
 \operatorname{sech}\bigl(\eta(x-x_*)\bigr).
\]
Substitution of the rates already derived in above completes the proof of case \textup{(iii)}--\textup{(vi)}. The formulas
for their multipliers are the corresponding formulas in those sections after
using \(m=M/N\) and \(\kappa=cN/2\).

It remains to justify the displayed determinant profiles for
families~\textup{(vii)}--\textup{(ix)}.  Starting from
\eqref{eq:det-solution} and using
\[
 \mathcal D_a^{-1}
 =F_a^{-1}\operatorname{adj}(\mathcal D_a)
\]
gives~\eqref{N3} entry by entry.
Multiplying the sparse matrices \(\mathcal A_a\) and \(C\) gives
\eqref{T2}; thus these are genuine finite algebraic
expressions, not implicit definitions through the differential equation.
The rates in Case (viii)-Case (ix) are obtained from
\eqref{eq:Cplus-first-rates}--\eqref{eq:Cplus-last-rate},
\eqref{eq:Cminus-first-rates}--\eqref{eq:Cminus-last-rate}, and
\eqref{eq:D-middle}--\eqref{eq:D-rate-one} by the same substitutions
\(m=M/N\) and \(\kappa=cN/2\).  The stated ranges are exactly the positivity
ranges already proved there, with the collision values removed.

Finally, Proposition~\ref{prop:mass-formula} verifies the two constraints for
every noncollision determinant profile, while
\eqref{eq:J-lambda-zero} verifies them for the two collision families.
Standing \eqref{eq9.4}, together with the exhaustive
rank and support analysis , proves that no additional
profiles exist.
\end{proof}
\section*{Acknowledgments}
X. Luo is supported by the National Natural Science Foundation of China (No. 11901147). Juncheng Wei was supported by National R and D Program of China (Grant
No. 2022YFA1005602), and Hong Kong General Research Fund "New frontiers in singular
limits of nonlinear partial differential equation". The
authors used OpenAI models as assistive tools in preparing this manuscript. Mathematical
arguments, proofs, and computations in the manuscript were verified and finalized by the
authors.

\end{document}